\documentclass[11pt,letter]{article}

\usepackage{amsthm,amsmath,epsfig,amsfonts,amssymb,bm,xcolor,amsbsy,times, dsfont, amssymb,amsmath,amscd,latexsym, amsthm, amsxtra,amsfonts}
\usepackage{bm}
\usepackage{bbm}
\usepackage{enumerate}
\usepackage{mathrsfs}
\usepackage{graphicx}
\usepackage{subfig}
\usepackage{comment}
\usepackage{mathtools}
\usepackage{cases}
\usepackage{dsfont, amssymb,amsmath,amscd,latexsym, amsxtra,amsfonts,amsthm}
\numberwithin{equation}{section}
\usepackage{ifpdf}
\usepackage{color}
\usepackage[title]{appendix}
\definecolor{webgreen}{rgb}{0,.5,0}
\definecolor{webbrown}{rgb}{.8,0,0}
\definecolor{emphcolor}{rgb}{0.95,0.95,0.95}
\usepackage{verbatim}
\usepackage{hyperref}
\hypersetup{hidelinks,
	colorlinks=true,
	linkcolor=blue,
	filecolor=blue,
	citecolor=blue,
	urlcolor=blue,
	breaklinks=true}

\usepackage{pgfplots}
\pgfplotsset{compat=1.17}

\newtheorem{theorem}{Theorem}[section]
\newtheorem{lemma}[theorem]{Lemma}

\newtheorem{assumption}{Assumption}[section]
\newtheorem{definition}[theorem]{Definition}

\theoremstyle{definition}
\newtheorem{example}[theorem]{Example}
\newtheorem{remark}[theorem]{Remark}
\numberwithin{equation}{section}

\DeclareMathOperator*{\argmax}{arg\,max}

\newcommand{\mR}{\mathbb{R}}

\newcommand{\mN}{\mathbb{N}}

\newcommand{\mE}{\mathbb{E}}
\newcommand{\mF}{\mathbb{F}}
\newcommand{\mP}{\mathbb{P}}

\renewcommand{\star}{\ast}
\renewcommand{\epsilon}{\varepsilon}

\newcommand{\cX}{\mathcal{X}}
\newcommand{\cY}{\mathcal{Y}}

\newcommand{\cA}{\mathcal{A}}
\newcommand{\cF}{\mathcal{F}}

\newcommand{\cP}{\mathcal{P}}

\title{Major-Minor Mean Field Game of Stopping: An Entropy Regularization Approach}
\author{ Xiang Yu\thanks{Department of Applied Mathematics,  The Hong Kong Polytechnic University, Kowloon, Hong Kong. Email:\url{xiang.yu@polyu.edu.hk}}
\and Jiacheng Zhang\thanks{Department of Statistic, The Chinese University of Hong Kong, New Territories, Hong Kong. Email:\url{jiachengzhang@cuhk.edu.hk}}
	\and Keyu Zhang\thanks{
	Department of Mathematical Sciences, Tsinghua University, Beijing , China. Email:\url{Zhangky21@mails.tsinghua.edu.cn}}
	\and Zhou Zhou\thanks{School of Mathematics and Statistics, University of Sydney, Sydney, Australia. Email:\url{zhou.zhou@sydney.edu.au}} }
\date{}

\begin{document}
	\maketitle
	
	\vspace{-0.3in}	
	\begin{abstract}
This paper studies a discrete-time major-minor mean field game of stopping where the major player can choose either an optimal control or stopping time. We look for the relaxed equilibrium as a randomized stopping policy, which is formulated as a fixed point of a set-valued mapping, whose existence is challenging by direct arguments. To overcome the difficulties caused by the presence of a major player, we propose to study an auxiliary problem by considering entropy regularization in the major player's problem while formulating the minor players' optimal stopping problems as linear programming over occupation measures. We first show the existence of regularized equilibria as fixed points of some simplified set-valued operator using the Kakutani–Fan–Glicksberg fixed-point theorem. Next, we prove that the regularized equilibrium converges as the regularization parameter $\lambda$ tends to 0 and the limit corresponds to a fixed point of the original operator, thereby confirming the existence of a relaxed equilibrium in the original problem.  {We also extend this entropy regularization method to the mean-field game problem where the minor players choose optimal controls.} 
		
		\vspace{0.1in}
	
	\vspace{0.1in}
	\end{abstract}

\vspace{0.1in}    
\section{Introduction}	

After the seminal studies of Lasry and Lions \cite{2007lasry} and Huang, Caines, and Malhamé \cite{2006huang}, mean field game (MFG) problems have garnered significant development during the past decades thanks to the merit of tractability and the broad applications in different fields such as finance, economics and biology. For a preliminary introduction on this topic, we refer to Cardaliaguet \cite{Pierre2013}; and more recent advances can be found in Carmona and Delarue \cite{book:carmona1,book:carmona2}. 

With the fast-growing theoretical advances in mean-field theories and stochastic control, different variations of MFG have been proposed and studied. Among them, one burgeoning direction is to study the MFG problem of optimal stopping, where all agents aim to choose the optimal stopping time under weak interactions with the population via the  distribution of the stopped state processes or the stopping times. MFG of optimal stopping was first studied by Nutz \cite{2018nutz} and by Carmona, Delarue, and Lacker \cite{2017Carmona}, employing a probabilistic approach in a general framework. In settings where the underlying process is described by a stochastic differential equation (SDE), MFG of optimal stopping has been examined by Bertucci \cite{2018BERTUCCI} using a PDE approach, and by Bouveret, Dumitrescu, and Tankov \cite{2020tankovsicon}, as well as Dumitrescu, Leutscher, and Tankov \cite{Peter2021}, using the linear programming approach based on occupation measures. 

Another important extension of MFG problems is to investigate the asymmetric interactions between a major player and a continuum of minor players, see some pioneer studies among \cite{bensoussan_mean_2016, carmona_alternative_2017,  Carmona2016,huang_large-population_2010, nguyen_linear-quadratic-gaussian_2012,dayanikli_machine_2024,aurell_optimal_2022}. This direction of extension addresses the limitation of the stringent assumption in classical MFG that, as the number of players grows large, the influence of any single player on the population becomes asymptotically negligible, which does not match with some practical applications when some decision maker's choices have prominent impacts on all other individual players even in the large population. However, as the minor player's objective function and state process depend on the state of the major player, the policy chosen by each minor player naturally need to be adapted to the filtration generated by the major player's state process as if there exists a common noise. Due to this reason, the presence of a major player spurs many theoretical challenges closely related to the well-known difficulties caused by common noise.

In the presence of common noise, some key mathematical obstacles arise such as the infinite dimensional space of stochastic measure flows hinders some standard compactification arguments and the fixed point arguments based on topological methods, see the detailed discussion in Carmona, Delarue, and Lacker \cite{2014Carmona} and their proposed discretization procedure of common noise. In the context of MFG of optimal stopping, incorporating common noise brings more challenges. The models in \cite{2017Carmona,2018nutz} allow for common noise and establish the existence of an equilibrium by exploiting the special structure of the game. However, these models have limitations and cannot be applied to certain types of games, such as those of the war of attrition type. To the best of our knowledge, the only work that considers a general framework for MFG of optimal stopping with common noise is \cite{Peter2023}, which, due to the aforementioned technical challenges, restricts their setup to a discrete-time setting and a finite state space to facilitate the linear programming approach. 

The present paper aims to investigate a major-minor MFG, in which the major player faces either an optimal control or an optimal stopping problem, while each minor player in the large population aims to select an optimal stopping time. For technical convenience, we also focus on a discrete-time setting and assume that the state space of the major player is finite.
The mean field feature is encoded by the fact that the transition kernels and reward functions of major and minor players depend on the law of the minor players. In the context of optimal stopping, this is specified by a flow of subprobability measures \( \mu_t \) for \( t = 0, 1, \dots, T \) and \( m_t \) for \( t = 0, 1, \dots, T-1 \), where \( \mu_t \) represents the distribution of minor players who have stopped at time \( t \), and \( m_t \) represents the distribution of minor players still active in the game at time $t$. Comparing with the classical MFG problems, each minor player’s transition kernel and reward function now depend on the major player’s state. As a direct consequence, the decision making of the minor players are influenced by a form of common noise stemming from the state of the major player, which renders the processes \( \mu \) and \( m \) stochastic and adapted to the filtration generated by the major player. In contrast to \cite{Peter2023}, this form of common noise is not an exogenous source of randomness, but influenced by the major player’s choice of policy. The coupled structure in the  equilibrium stopping policies of major and minor players introduces additional complexity and remains underexplored in the existing literature.

In this study, we prove the existence of a relaxed equilibrium (see Definition \ref{def:relaxed}) for a major-minor MFG of stopping, allowing for randomized policies for both the major and minor players. Contrary to the compactification technique used in the literature of MFG problems (see e.g., \cite{2014Carmona, 2017Carmona,  Peter2021,lacker_mean_2015}), the non-uniqueness of the solutions to the major player's problem leads to a set-valued  mapping that is not convex-valued in general, which hinders the direct application of the Kakutani-Fan-Glicksberg theorem, see the discussion below Definition \ref{def:relaxed}. In response to this challenge, instead of working with the original problem directly, we propose to consider the auxiliary formulation of the major player's problem under entropy regularization, in which the entropy term of a randomized policy is added into the reward functional of the major player (see Definition \ref{def:regul}). Drawing the merit from the entropy regularization formulation that the optimal solution for the major player is unique and admits an explicit form of Gibbs-measure, we are allowed to exercise a linear programming approach to tackle the minor players' optimal stopping problem by considering a simplified auxiliary set-valued mapping. The existence of a regularized equilibrium (see Definition \ref{def:regul}) is then established by applying the Kakutani–Fan–Glicksberg fixed-point theorem for the auxiliary mapping; see Theorem \ref{them:regular_existence}. Next, we approximate the original problem by a sequence of its entropy-regularized counterparts with a vanishing regularization parameter $\lambda \rightarrow 0$, and show that the limit of the sequence of regularized equilibria is a fixed point of the original set-valued operator and it is therefore a relaxed equilibrium in the original major-minor MFG; see Theorem \ref{thm:existencerelaxed}. { Moreover, to enhance the applicability of the entropy regularization method in a more general context, we also study in Section \ref{sect:6} a class of major-minor MFG where the minor players choose optimal controls, for which we extend the method of the vanishing entropy regularization to confirm the existence of relaxed equilibrium. 
}

To summarize, the contributions of the present paper are three-fold: Firstly, our work, to the best of our knowledge, is the first attempt to study major-minor MFG in the context of optimal stopping, whereas existing literature mainly focuses on optimal control problems. Secondly, we extend the linear programming framework (see for instance \cite{Peter2023,guo2022optimization,guo2024mf}) to a more general setting by incorporating the major player's influence on all minor players' decision making and allowing for measure-dependent transition kernels. Thirdly, we develop a method by combining the entropy regularization and linear programming techniques to establish the existence of relaxed equilibria in major-minor MFG based on the vanishing regularization arguments. We highlight that the entropy regularization plays a key role to select a unique solution to the major player's problem, which successfully eliminates the technical issues arising from the direct fixed point argument for the original problem. Note that, although the entropy regularization has been recently applied to facilitate the existence of fixed point in some time-inconsistent control problems (see, e.g., \cite{erhan2024, yu2024}), our attempt to fuse this idea into 
the proof of existence of relaxed equilibrium for major-minor MFG of stopping is novel and may motivate some related future studies.  

The rest of the paper is organized as follows. Section \ref{sect:pf} introduces the problem formulation of the major-minor MFG of stopping and the definition of relaxed equilibria. In addition, the auxiliary problem under entropy regularization is also formulated therein together with the definition of regularized equilibria. Section \ref{sect:technical assumptions} details the technical assumptions required for our analysis. 
The existence of regularized equilibria for the auxiliary problem is proved in Section \ref{sect:regul}. The existence of relaxed equilibria for the original MFG problem is established in Section \ref{sect:relaxed} via vanishing entropy regularization. 
 {In Section \ref{sect:6}, we generalize the vanishing entropy regularization method to solve a major-minor MFG where minor players face optimal control problems. Finally, the paper ends with an appendix consisting of two parts: Appendix \ref{appendix:proof} contains the proofs of several lemmas in the main body of the paper, and Appendix \ref{appendixB} presents a related auxiliary result.}

\paragraph{Notations}
Throughout this paper, we equip any finite set $E$,  with the discrete metric $d^0$. For a topological space \( (E, \tau) \), we denote by \( \mathcal{B}(E) \) the Borel \( \sigma \)-algebra, by \( \mathcal{M}^s(E) \) the set of Borel finite signed measures on \( E \), by \( \mathcal{M}(E) \) the set of Borel finite positive measures on \( E \), by \( \mathcal{P}^{\text{sub}}(E) \) the set of Borel subprobability measures on \( E \) and by \( \mathcal{P}(E) \) the set of Borel probability measures on \( E \). For any $\rho\in \cP(E)$, we use  \( \operatorname{supp}(\rho) \) to stand for the support of \( \rho\).  We denote by \( M(E) \) the set of Borel measurable functions from \( E \) to \( \mathbb{R} \), by \( M_b(E) \) the subset of Borel measurable and bounded functions, by \( C(E) \) the subset of continuous functions, and by \( C_b(E) \) the subset of continuous and bounded functions. The set \( M_b(E) \) is endowed with the supremum norm \( \|\varphi\|_{\infty} = \sup_{x \in E} |\varphi(x)| \).   We use $\cY^\cX$ to denote the set of functions from $\cX$ to $\cY$ and use $\delta_{x}$ to denote the Dirac measure at $x$. { We use $P(de)$ to denote the probability measure on a topological space \( (E, \tau) \). When \( E \) is finite, we sometimes write it as $P(e')$, which should not cause any confusion.}

\section{Problem Formulation}\label{sect:pf}
 Let \( \mathbb{T} = \{0, 1, \dots, T\} \) be the set of discrete time indices and \( \mathbb{T}^* = \{1, 2, \dots, T\} \). We consider a stochastic process \((X^0, X)\) on a filtered probability space \( (\Omega, \mathcal{F}, \mathbb{F}, \mathbb{P}) \), where \( X^0 = (X^0_t)_{t \in \mathbb{T}} \)  represents the major player's state, taking values in \( S^0 \), a finite set with a discrete metric \( d^0 \), and \( X = (X_t)_{t \in \mathbb{T}} \) stands for the state of the minor players, taking values in a nonempty compact metric space \((S, d)\). 

Let \( A \) be a compact action space for the major player, with two primary cases of interest: (1) \( A \subset \mathbb{R}^\ell \), for a fixed \( \ell \in \mathbb{N} \) with \( \text{Leb}(A) > 0 \); and (2) \( A = \{0, 1\} \)\footnote{The argument in this paper can be directly extended to the case where \( A \) is a finite set.}, which covers to the scenario of an optimal stopping. 
All minor players in the large population need to solve optimal stopping problems. The precise formulations of the optimization problems for major and minor players will be given later.

In this paper, we allow the major player's state process \( X^0 \) to be non-Markovian by lifting the state space to the space of all trajectories, a technique also employed in \cite{Peter2023,kurtz_martingale_1998}. The trajectory of \( X^0 \) up to time \( t \) is described by a finite-dimensional matrix-valued process \( U_t \). Define \(\mathcal{I} := \{1, \dots, |S^0|\}\), where \( S^0 = \{z_j : j \in \mathcal{I}\} \) is the finite state space of \( X^0 \), with distinct elements \( z_j \neq z_{j'} \) for \( j \neq j' \), \((j, j') \in \mathcal{I}^2\). The process \( U \), taking values in \( W := \{0,1\}^{(T+1) \times |S^0|} \) (the space of \((T+1) \times |S^0|\)-dimensional binary matrices), is defined as follows:  
\begin{equation}\label{path}
    U_t(\omega)(s+1, j) = 1_{s \leq t} 1_{X^0_s(\omega) = z_j}, \quad (s, j) \in \mathbb{T} \times \mathcal{I}, \quad t \in \mathbb{T}, \quad \omega \in \Omega,
\end{equation}
where \( \mathbb{T} = \{0, \dots, T\} \).
Alternatively, for \( s \in \mathbb{T} \) and \( z \in S^0 \), let \( M[s, z] \) denote the \((T+1) \times |S^0|\)-dimensional matrix given by
\[
M[s, z](r+1, j) = 1_{s = r} 1_{z = z_j}, \quad (r, j) \in \mathbb{T} \times \mathcal{I}.
\]
Then, for \( t \in \mathbb{T} \), we have 
\begin{equation}\label{UX}
	 U_t = \sum_{s=0}^t M[s, X^0_s].
\end{equation}
Define a function \( \Psi_t: W \rightarrow S^0 \) by
\[
\Psi_t(u) := \sum_{j \in \mathcal{I}} z_j 1_{u(t+1, j) = 1}, \quad u \in W.
\]
It is easy to see that
\begin{equation}\label{XU}
	 X^0_s = \Psi_s(U_t),\quad \forall s \leq t\,\ \text{and}\,\  t \in \mathbb{T} ,
\end{equation}
 hence  the information about
the entire trajectory of $X^0$ up to time $t$ is encoded  in \( U_t \).  Moreover, we can conclude that $\sigma(U_t)=\cF^{X^0}_{t}=\cF^{U}_{t}$.

\subsection{Transition Law}
We next specify the transition law of the major and minor players in our setting. 
The initial states are drawn from distributions \(\tilde{P}^0_0 \in \mathcal{P}(S^0)\) and \(P_0 \in \mathcal{P}(S)\). For \(k \in \mathbb{T} \setminus \{T\}\), the transitions are governed by transition kernels:
\begin{align*}
	&X_{k+1}^0 \sim \tilde{P}^0_{k+1}(U_k, \alpha_k, m_k, dx^0), \\
	&X_{k+1} \sim P_{k+1}(X_k, U_k, m_k, dx).
\end{align*}
Here, \(\alpha_k: W \to \mathcal{P}(A)\) is the major player's relaxed feedback control, and \( \mathcal{A} \) denotes the set of all such control sequences. The term \(m_k: W \to \mathcal{P}^{\text{sub}}(S)\) is the mean-field interaction term, representing the conditional distribution of minor players that remain active at time \(k\), given the information of the major player \(\mathcal{F}^U_k\).

To overcome the issue of non-Markovian assumption, we work with \( U \) rather than \( X^0 \). By \eqref{UX}, we obtain  
\begin{align*}
	&U_{k+1} \sim P^0_{k+1}\left( U_k, \alpha_k, m_k, u' \right):=\sum_{x^0\in S}1_{U_k+M(k+1,x^0)=u'}\tilde{P}^0_{k+1} \left( U_k, \alpha_k, m_{k}, x^0\right), \quad k \in \mathbb{T} \setminus \{T\}, \\& U_0 \sim P_0^0(u):=\tilde{P}^0_0(M(0,X_0^0)=u).
\end{align*}
 It is evident that all properties of \(\tilde{P}^0(u,a,m,u')\) with respect to the variables \((a,m)\) are inherited from \(P^0\). 
  The dynamics of the joint process \((U,X)\) under a relaxed control \(\alpha \in \mathcal{A}\) and mean-field term \(m\) are thus described by the kernels:
  \begin{equation}\label{transition1}
  	\begin{split}
  			&U_{k+1}\sim \pi^0_{k+1}(U_{k},du';\alpha_k(U_{k}),m_k(U_{k})):= \int_{A}P^0_{k+1}(U_{k},a,m_{k}(U_{k}),du')\alpha_{k}(U_{k})(da),\\
  		&X_{k+1}\sim P_{k+1}(X_{k},U_{k},m_{k}(U_{k}),dx').
  	\end{split}
  \end{equation}
\subsection{Objective Functionals}
The major player's objective is to maximize a reward functional that depends on their actions and two mean-field terms: \( m = (m_k)_{k \in \mathbb{T}\setminus\{T\}} \) representing the conditional distribution of  active minor players, and \( \mu = (\mu_k)_{k \in \mathbb{T}} \) representing the conditional distribution of stopped minor players. For an admissible control policy $\alpha = (\alpha_t)_{t \in \mathbb{T}\setminus\{T\}}\in\mathcal{A}$, the reward functional is
	\begin{equation*}
J^0(\alpha; \mu, m) = \mathbb{E} \left[ \sum_{t=0}^{T-1} \int_{A} f_{t}^0(\Psi_t(U_{t}), a, m_t(U_t)) \, \alpha_{t}(U_t)(da) + g^0(\Psi_T(U_T), \mu_T(U_T)) \right].
\end{equation*}
Here $f^0_t : S^0 \times A \times \cP^{\mathrm{sub}}(S) \to \mathbb{R}$ denotes the running reward at time $t$, and $g^0 : S^0 \times \cP^{\mathrm{sub}}(S) \to \mathbb{R}$ denotes the terminal reward.

Notably, this framework can accommodate the optimal stopping feature with additional assumptions. To wit, we introduce an enlarged state space \( \bar{S}^0 = S^0 \cup \{\triangle\} \), where \( \triangle \) represents the stopped state. We also consider the action space \( A = \{0, 1\} \), where \( a = 1 \) indicates the immediate stopping and \( a = 0 \) indicates the continuation. In this setup, we assume that the functions \( \tilde{P}^0 \), \( f^0 \), and \( g^0 \) satisfy the following conditions:
\begin{equation*}
	\begin{split}
	    &\operatorname{supp}(\tilde{P}_{t+1}^0(u, 0, m,dx))\in S^0, \quad \forall (t, u, m) \in \mathbb{T}\setminus\{T\} \times W \times \mathcal{P}^{\text{sub}}(S),\\
		&\tilde{P}_{t+1}^0(u, 1, m,dx) = \delta_{\triangle}, \quad \forall (t, u, m) \in \mathbb{T}\setminus\{T\} \times W \times \mathcal{P}^{\text{sub}}(S),\\
		&\tilde{P}_{t+1}^0(u, a, m,dx) = \delta_{\triangle}, \quad \forall (t, a, m) \in \mathbb{T}\setminus\{T\} \times A \times \mathcal{P}^{\text{sub}}(S)\,\text{and}\, \Psi_t(u)=\Delta,\\
		&{f^0_{t}(\triangle, 1, m) \geq f^0_{t}(\triangle, 0, m), \quad \forall (t, m) \in \mathbb{T}\setminus\{T\} \times \mathcal{P}^{\text{sub}}(S)},\\
		&{g^0(x, \mu)=0, \quad \forall (x, \mu) \in \bar{S}^0 \times \mathcal{P}^{\text{sub}}(S)\text{\footnotemark},}\\
	\end{split}
\end{equation*}
where  the first two conditions ensure that the major player exits the game (i.e., reaches state \( \Delta \)) if and only if the action \( a = 1 \). The third condition specifies that once the major player exits the game, it is irreversible (i.e., \( \Psi_t(u) = \Delta \) implies the process stays at \( \Delta \)). {The fourth condition incentivizes the major player such that, if she exits the game before the terminal period \( T \), she will not change her action again.} 
In this scenario, the  stopping policy is given by $\tau^{0}:=\inf\{t\in \mathbb{T}: \alpha_t=1\}\wedge T$ with the convention that $\inf\emptyset=\infty$.
\footnotetext{Here, for ease of presentation, we set $g = 0$. In fact, alternative assumptions can be made to account for cases where the major player receives nonzero payoffs when she decides to stop, as our proof does not depend on these specific assumptions.}

The reward for a representative minor player depends on their chosen stopping time \( \tau \in \mathcal{T} \), where \(\mathcal{T}\) is the set of \(\mathbb{F}\)-stopping times valued in \(\mathbb{T}\). The objective is to maximize:
\begin{equation*}
 J^1(\tau; \mu, m) = \mathbb{E} \left[ \sum_{t=0}^{\tau - 1} f_t(X_t, U_t, m_t(U_t)) + g_{\tau} \left( X_{\tau}, U_{\tau}, \mu_{\tau}(U_{\tau}) \right) \right].
\end{equation*}
Here
\(
    f_t : S \times W \times \mathcal{P}^{\text{sub}}(S) \to \mathbb{R}
\) denotes the continuation reward at time \(t\),
and 
\(
    g_t : S \times W \times \mathcal{P}^{\text{sub}}(S) \to \mathbb{R}
\)
denotes the stopping reward at time \(t\).
\subsection{Linear Programming Formulation for the Minor Player's Problem}
The minor player's optimal stopping problem can be reformulated using a linear programming approach over occupation measures, following \cite{2020tankovsicon}. First, we can conclude from \eqref{transition1} that the joint process \((X,U)\) is a  Markov chain on \(S \times W\) with the transition kernels 
\begin{equation}\label{transitionkernal}
    \pi(x,u;\alpha,m;dx',du'):=\left(\pi_t(x,u;\alpha_{t-1}(u),m_{t-1}(u);dx',du')\right)_{t\in \mathbb{T}^*}
\end{equation} 
specified by 
\begin{equation*}
	\pi_t(x,u;\alpha_{t-1}(u),m_{t-1}(u);dx',du')=P_{t}(x,u,m_{t-1}(u),dx')\pi^0_{t}(u,du';\alpha_{t-1}(u),m_{t-1}(u)).
\end{equation*}
 Let \( (\Omega, \mathcal{F}, \mathbb{F}, \mathbb{P}) \) be a complete filtered probability space such that  $(X,U)$ is an $\mF$-Markov chain with these kernels.
For any \(\mathbb{F}\)-stopping time \(\tau\) valued in \(\mathbb{T}\), we define two sequences of occupation measures:
\begin{align*}
	&\tilde{m}_t(B\times \{u\}) := \mathbb{P}[X_t \in B, U_t=u, t < \tau ],\quad  \quad B \in \mathcal{B}(S),\quad u\in W, \quad t \in \mathbb{T} \setminus \{T\},\\
	&\tilde{\mu}_t(B\times \{u\}) := \mathbb{P}[X_t \in B,U_t=u,\tau= t ], \quad  \quad B \in \mathcal{B}(S), \quad u\in W,  \quad t \in \mathbb{T}.
\end{align*}

\subsubsection{Derivation of the constraint}\label{Derivation of the constraint}
As $(X,U)$ is an $\mF$-Markov chain with transition kernels $\pi(x,u;\alpha,m; dx',du')$, for any $\varphi \in C_b(\mathbb{T} \times S \times W)$ we have
\[
\mathbb{E}[\varphi(t + 1, X_{t+1}, U_{t+1}) - \varphi(t, X_t, U_t) | \mathcal{F}_t] = \mathcal{L}(\varphi)(t, X_t, U_t;\alpha_{t}(U_t),m_{t}(U_t)), \quad \text{a.s.}
\]
where for $(t, x, u) \in \mathbb{T} \setminus \{T\} \times S \times W$,
\[
\mathcal{L}(\varphi)(t, x, u;\alpha_{t}(u),m_{t}(u)) := \int_{S \times W} [\varphi(t + 1, x', u') - \varphi(t, x, u)] \pi_{t+1}(x, u;\alpha_{t}(u),m_{t}(u); dx', du').
\]
Moreover, the process $\mathcal{M}(\varphi)$ defined by $\mathcal{M}_0(\varphi) = \varphi(0, X_0, U_0)$ and
\[
\mathcal{M}_t(\varphi) = \varphi(t, X_t, U_t) - \sum_{s=0}^{t-1} \mathcal{L}(\varphi)(s, X_s, U_s;\alpha_{s}(U_s),m_{s}(U_s)), \quad t \in \mathbb{T}^*,
\]
is an $\mathbb{F}$-martingale. In particular, for any \(\mathbb{F}\)-stopping time $\tau$, the Optional Sampling Theorem implies
\[
\mathbb{E}[\varphi(\tau, X_\tau, U_\tau)] = \mathbb{E}[\varphi(0, X_0, U_0)] + \mathbb{E} \left[ \sum_{t=0}^{\tau-1} \mathcal{L}(\varphi)(t, X_t, U_t;\alpha_{t}(U_t),m_{t}(U_t)) \right].
\]
We deduce that the occupation measures satisfy the following constraint
\begin{align*}
	\sum_{t=0}^{T} \int_{S\times W} \varphi(t, x, u) \tilde{\mu}_t(dx,du)
	&= \int_{S\times W} \varphi(0, x, u) m_0^*(dx,du) \\&+ \sum_{t=0}^{T-1}  \int_{S\times W} \mathcal{L}(\varphi)(t, x, u;\alpha_t(u),m_{t}(u)) \tilde{m}_t(dx,du),
\end{align*}
for any $\varphi \in C_b(\mathbb{T} \times S \times W)$,  
where \(m_0^*(B \times \{u\}) \coloneqq \mathbb{P}(X_0 \in B, U_0 = u)\) is the initial joint distribution.

For notational convenience, let us define sets \( \mathcal{C}_\mu := \prod_{t \in \mathbb{T}} \mathcal{P}^{\text{sub}}(S)^{W} \), \( \mathcal{C}_m := \prod_{t \in \mathbb{T} \setminus \{T\}} \mathcal{P}^{\text{sub}}(S)^{W} \) and $\mathcal{V}:=\prod_{t\in\mathbb{T}}\cP^{sub}(S\times W)\times\prod_{t\in\mathbb{T}\setminus\{T\}}\cP^{sub}(S\times W)$.
{\begin{definition}[Set of Admissible Occupation Measures]\label{constraint}
	 For a fixed tuple \( (\mu, m, \alpha) \in \mathcal{C}_{\mu} \times \mathcal{C}_m \times \mathcal{A} \), the set of admissible occupation measures, \(\mathcal{R}[m; \alpha]\), is the set of pairs \( (\tilde{\mu}, \tilde{m}) \in \mathcal{V} \) such that for all test functions \( \varphi \in C_b(\mathbb{T} \times S \times W) \),
		\begin{align*}
			\sum_{t=0}^{T}\int_{S\times W} \varphi(t, x, u) \, \tilde{\mu}_t(dx,du)
			&=  \int_{S\times W} \varphi(0, x, u) \, m_0^*(dx,du) \\ &+ \sum_{t=0}^{T-1} \int_{S\times W} \mathcal{L}(\varphi)(t, x, u; \alpha_t(u), m_t(u)) \, \tilde{m}_t(dx,du),
		\end{align*}
    where \(m_0^*(B \times \{u\}) \coloneqq \mathbb{P}(X_0 \in B, U_0 = u)\) is the initial joint distribution.
\end{definition}}
{\begin{remark}
    For any fixed tuple \( (\mu, m, \alpha) \in \mathcal{C}_{\mu} \times \mathcal{C}_m \times \mathcal{A} \), it follows from Theorem~18 in \cite{Peter2023} that the set $\mathcal{R}[m;\alpha]$ characterizes all occupation measures associated with randomized stopping times, in the sense that for every  $(\tilde{\mu}, \tilde{m})\in\mathcal{R}[m;\alpha]$, there exists a complete filtered probability space $( \bar{\Omega}, \overline{\mathcal{F}}, \overline{\mathbb{F}}, \overline{\mathbb{P}})$ supporting the
    random variables $(\bar{\tau}, \bar{X}, \bar{U})$ such that 
    \begin{itemize}
        \item $(\bar{X}, \bar{U})$ is an $\overline{\mathbb{F}}$-Markov chain valued in $S \times W$ with transition kernels $\pi$ defined in \eqref{transitionkernal} and initial distribution $m_0^*$.
        \item $\bar{\tau}$ is an $\overline{\mathbb{F}}$-stopping time valued in $\mathbb{T}$.
        \item  The measures admit the following representations:
        \[
        \begin{aligned}
        & \tilde{m}_t(B\times \{u\})=\overline{\mathbb{P}}\!\left(\bar{X}_t \in B,\,  \bar{U}_t=u,  t<\bar{\tau} \right), \quad B \in \mathcal{B}(S), \; u \in W, \; t \in \mathbb{T}\setminus\{T\},  \\
        & \tilde{\mu}_t(B\times \{u\})=\overline{\mathbb{P}}\!\left(\bar{X}_t \in B,\, \bar{U}_t=u,\, \bar{\tau}=t \right), \quad B \in \mathcal{B}(S),  \; u \in W,\; t \in \mathbb{T}.
        \end{aligned}
        \]
    \end{itemize}
\end{remark}
}The minor player's problem can then be expressed as a linear programming problem over the set of admissible occupation measures.

{\begin{definition}[LP optimization criteria]\label{def:lpop}
	For \((\mu, m, \alpha) \in \mathcal{C}_{\mu} \times \mathcal{C}_m \times \mathcal{A}\), the reward functional for the minor player, \(J^1\), associated with an occupation measure pair \((\tilde{\mu}, \tilde{m}) \in \mathcal{R}[m;\alpha]\), is defined by
	\begin{align*}
		J^1(\tilde{\mu}, \tilde{m}; \mu, m) \coloneqq \sum_{t=0}^{T-1} \int_{S\times W} f_t(x, u, m_t(u)) \, \tilde{m}_t(dx,du) + \sum_{t=0}^{T} \int_{S\times W} g_t(x, u, \mu_t(u)) \, \tilde{\mu}_t(dx,du).
	\end{align*}
\end{definition}}
\subsection{Major-Minor mean field equilibrium}
Our goal is to examine the existence of the relaxed equilibrium for the major-minor mean field game. This is achieved by formulating the equilibrium conditions as a fixed-point problem.

To begin, we formally define the marginal law of the major player's process, which plays a central role in the consistency condition.
 For any given $(\alpha,m) \in \mathcal{A}\times \mathcal{C}_m$, the distribution $p_t(\cdot; \alpha, m) \in \mathcal{P}(W)$ of the major player's state $U_t$ can be determined recursively as
\begin{align*}
    &p_0(du; \alpha, m) = \mathbb{P}(U_0 \in du), \\
    &p_{t+1}(du; \alpha, m) = \int_{W} \pi^0_{t+1}(u', du; \alpha_t(u'), m_t(u')) \, p_t(du'; \alpha, m), \quad \text{for } t \in \mathbb{T}\setminus\{T\}.
\end{align*}
{\begin{definition}[Relaxed Equilibrium]\label{def:relaxed}
A relaxed equilibrium is a tuple \((\mu^*, m^*, \alpha^*, (\tilde{\mu}^*, \tilde{m}^*))\) that satisfies the following three conditions simultaneously:

\begin{enumerate}
    \item  Given the mean-field terms \((\mu^*, m^*)\), the major player's relaxed control \(\alpha^* \in \mathcal{A}\) attains the optimality: 
    \[ \alpha^* \in \mathbb{A}(\mu^*, m^*) \coloneqq \operatorname*{arg\,max}_{\alpha \in \mathcal{A}} J^0(\alpha; \mu^*, m^*). \]

    \item  Given \((\mu^*, m^*, \alpha^*)\), the  occupation measure pair \((\tilde{\mu}^*, \tilde{m}^*) \in \mathcal{V}\) attains the optimality: 
    \[ (\tilde{\mu}^*, \tilde{m}^*) \in \hat{\Theta}(\mu^*, m^*, \alpha^*) \coloneqq \operatorname*{arg\,max}_{(\tilde{\mu}, \tilde{m}) \in \mathcal{R}[m^*; \alpha^*]} J^1(\tilde{\mu}, \tilde{m}; \mu^*, m^*). \]

    \item  The mean-field terms \((\mu^*, m^*)\) must coincide with the conditional distribution of the stopped and active minor players, given the major player's information. This consistency is defined via disintegration with respect to the major player's law, $p_t^*(\cdot) \coloneqq p_t(\cdot; \alpha^*, m^*)$, such that
    \begin{align*}
        \tilde{\mu}^*_t(dx, du) &= \mu^*_t(u)(dx) \, p^*_t(du) \quad \text{for } t \in \mathbb{T}, \\
        \tilde{m}^*_t(dx, du) &= m^*_t(u)(dx) \, p^*_t(du) \quad \text{for } t \in \mathbb{T}\setminus\{T\}.
    \end{align*}
\end{enumerate}
To recast these conditions concisely as a fixed-point problem, we introduce the \emph{conditioning mapping} $\Gamma: \mathcal{V} \times \mathcal{A} \times \mathcal{C}_m \to 2^{\mathcal{C}_\mu \times \mathcal{C}_m}$. Precisely, for any inputs $((\tilde{\mu}, \tilde{m}), \alpha, m)$, we define $\Gamma((\tilde{\mu}, \tilde{m}), \alpha, m)$ as the set of all pairs $(\mu', m')$ such that:
\begin{align*}
    \tilde{\mu}_t(dx, du) &= \mu'_t(u)(dx) \, p_t(du; \alpha, m) \quad \text{for } t \in \mathbb{T}, \\
    \tilde{m}_t(dx, du) &= m'_t(u)(dx) \, p_t(du; \alpha, m) \quad \text{for } t \in \mathbb{T}\setminus\{T\}.
\end{align*}
The equilibrium fixed-point mapping $\Phi: \mathcal{C}_\mu \times \mathcal{C}_m \to 2^{\mathcal{C}_\mu \times \mathcal{C}_m}$ is then defined by:
\[
    \Phi(\mu, m) \coloneqq \bigcup_{\alpha \in \mathbb{A}(\mu, m)} \bigcup_{(\tilde{\mu}, \tilde{m}) \in \hat{\Theta}(\mu, m, \alpha)} \Gamma\big((\tilde{\mu}, \tilde{m}), \alpha, m\big).
\]
A pair of mean field flows \((\mu^*, m^*)\) induces a relaxed equilibrium if and only if it is a fixed point of this mapping, i.e., \((\mu^*, m^*) \in \Phi(\mu^*, m^*)\).
\end{definition}}
The existence of a fixed point for the mapping $\Phi$ is typically established via the Kakutani-Fan-Glicksberg's fixed point theorem; see, for example, Corollary 17.55 in \cite{aliprantis06}, which requires the mapping to have nonempty \emph{convex values}. The set $\hat{\Theta}(\mu, m, \alpha)$ is convex for any fixed $(\mu, m, \alpha)$ due to the linearity of the objective functional $J^1$. Furthermore, as we will prove in Lemma~\ref{lem:properties_gamma_theta} below, the image of this set under the conditioning mapping, $\Gamma(\hat{\Theta}(\mu, m, \alpha), \alpha, m)$, is also convex. This convexity is a crucial property that holds whenever the major player's strategy is fixed.

However, the fixed-point mapping $\Phi$ may  fail to be convex-valued. The source of this failure is the potential non-uniqueness of the major player's best response $\alpha \in \mathbb{A}(\mu, m)$. Since $\Phi$ is constructed by taking the union of these convex response sets over all such optimal controls $\alpha$, this operation can yield a non-convex image, as the following example demonstrates.

\begin{example}
We construct a simple stopping game without mean field interaction to illustrate this issue, focusing on the crucial role of the conditioning mapping $\Gamma$.

\begin{enumerate}
    \item \textbf{Setup: The Game and Major Player's Strategy.}
    Let $T=2$. The major player (MP) chooses a stopping time $\tau^0 \in \{0, 1, 2\}$ to maximize their reward $G_{\tau^0}$, where $G_0 = 0$, $G_1 = 1$, and $G_2 = 1$. The minor player (mP) chooses $\tau \in \{0, 1, 2\}$ to maximize her reward $H_\tau$, where $H_0 = 1/2$, $H_1 = \mathbf{1}_{\{\tau^0=1\}}$, and $H_2 = 1/3$. There are no mean-field interactions.

    The MP's optimal reward is 1. An optimal relaxed strategy for the MP corresponds to a probability distribution $(a_0, a_1, a_2)$ over the stopping times $\{0, 1, 2\}$. To achieve the maximum reward, the MP must stop at $t=0$ with probability zero ($a_0=0$). Any distribution of the form $(0, p, 1-p)$ for $p \in [0,1]$ is optimal. The set of optimal controls is thus $\mathbb{A} = \{\alpha^p \mid p \in [0,1]\}$, where $\alpha^p$ induces the stopping distribution $(0, p, 1-p)$.

    This set of controls generates three potential MP paths:
    \begin{itemize}
        \item Path $u^0$: MP stops at $t=0$. (Occurs with probability $a_0=0$ under any optimal $\alpha$).
        \item Path $u^1$: MP continues at $t=0$, stops at $t=1$. (Occurs with probability $p$).
        \item Path $u^2$: MP continues at $t=0$, stops at $t=2$. (Occurs with probability $1-p$).
    \end{itemize}

    \item \textbf{The Role of $\Gamma$:}
   The mP's policy, represented by the output of \(\Gamma\), is adapted to the realization of the major player's strategy;
 specifically, it is a function \(u \mapsto \mu(u)\) that specifies the mP's conditional stopping distribution $(\mu_0, \mu_1, \mu_2)$ for each potential MP path $u \in \{u^0, u^1, u^2\}$.

    \item \textbf{Constructing Two Distinct Policies in $\Phi$.} We exploit the freedom in defining policies on zero-probability paths.
    \begin{itemize}
        \item \textbf{Case 1: MP plays $\alpha^0$ (deterministic stop at $t=2$).}
        This is an optimal strategy for the MP ($p=0$). The MP's marginal law is a point mass on path $u^2$. On this realized path, the mP's rewards are $(1/2, 0, 1/3)$, so their unique optimal action is to stop at $t=0$. The conditioning $\Gamma$ fixes the policy on path $u^2$ to be $\mu(u^2) = (1,0,0)$. The policies on the zero-probability paths $u^0$ and $u^1$ are unconstrained. We make a specific choice for them: $\mu(u^0) = \mu(u^1) = (1,0,0)$. This yields a complete policy function $\mu^*$ which is a valid best response, so $\mu^* \in \Phi$.

        \item \textbf{Case 2: MP plays $\alpha^1$ (deterministic stop at $t=1$).}
        This is also an optimal strategy ($p=1$). The MP's marginal law is a point mass on $u^1$. On this path, the mP's rewards are $(1/2, 1, 1/3)$, so their unique optimal action is to stop at $t=1$. $\Gamma$ fixes $\mu(u^1) = (0,1,0)$. The policies on the zero-probability paths $u^0$ and $u^2$ are unconstrained. We choose them to be $\mu(u^0) = \mu(u^2) = (0,1,0)$. This yields a second policy $\mu^{**}$, which is also in $\Phi$.
    \end{itemize}

    \item \textbf{The Convex Combination and Contradiction.}
    Both $\mu^*$ and $\mu^{**}$ belong to $\Phi$. If $\Phi$ were convex,  $\hat{\mu} \coloneqq \frac{1}{2}\mu^* + \frac{1}{2}\mu^{**}$ must also be in $\Phi$. This combined policy prescribes a specific stopping distribution on each path:
    \[ \hat{\mu}(u^0) = (1/2, 1/2, 0), \quad \hat{\mu}(u^1) = (1/2, 1/2, 0), \quad \hat{\mu}(u^2) = (1/2, 1/2, 0). \]
    For $\hat{\mu}$ to be in $\Phi$, it must be a best response to some optimal MP strategy $\alpha^p$ for $p \in [0,1]$.
    
    Consider any $\alpha^p$ with $p \in (0,1)$. Under this strategy, paths $u^1$ and $u^2$ occur with positive probability, while $u^0$ occurs with zero probability. An optimal adapted policy must prescribe an optimal action on every path that can occur.
    \begin{itemize}
        \item Along path $u^1$ (which occurs with prob. $p > 0$), the optimal mP action is to stop at $t=1$, requiring the policy component $(0,1,0)$.
        \item Along path $u^2$ (which occurs with prob. $1-p > 0$), the optimal mP action is to stop at $t=0$, requiring the policy component $(1,0,0)$.
    \end{itemize}
    The candidate policy $\hat{\mu}$ is suboptimal on both of these realized paths, since $\hat{\mu}(u^1) \neq (0,1,0)$ and $\hat{\mu}(u^2) \neq (1,0,0)$. Therefore, $\hat{\mu}$ is not a best response to any mixed strategy $\alpha^p$ with $p \in (0,1)$. (It is also not a best response to the deterministic strategies $\alpha^0$ or $\alpha^1$).

    Since $\hat{\mu}$ is not in $\Phi$, the set $\Phi$ is not convex.
\end{enumerate}
\end{example}
This failure prevents a direct application of Kakutani-Fan-Glicksberg's fixed point theorem and is the primary motivation for introducing entropy regularization.
Despite the potential non-convexity of $\Phi$, the inner component of the mapping, corresponding to a fixed major player strategy, does possess the desired convexity property. 
{\begin{lemma}\label{lem:properties_gamma_theta}
For any fixed tuple \((\mu, m, \alpha)\) such that the set \(\hat{\Theta}(\mu, m, \alpha)\) is nonempty, the image set
    \[ \Gamma\big(\hat{\Theta}(\mu, m, \alpha), \alpha, m\big) \coloneqq \bigcup_{(\tilde{\mu}, \tilde{m}) \in \hat{\Theta}(\mu, m, \alpha)} \Gamma\big((\tilde{\mu}, \tilde{m}), \alpha, m\big) \]
    is a nonempty convex subset of \(\mathcal{C}_\mu \times \mathcal{C}_m\).
\end{lemma}
\begin{proof}
The proof is postponed to Appendix~\ref{appendix:proof}.
\end{proof}}

The non-convexity of the fixed-point mapping $\Phi$ posed a significant obstacle. A potential approach to circumvent this is to select a single best response $\alpha$ from the set $\mathbb{A}(\mu,m)$ and analyze the fixed points of the resulting (now convex-valued) mapping $\Phi_\alpha(\mu,m) = \Gamma(\hat{\Theta}(\mu,m,\alpha),\alpha,m)$. However, it is unclear how to select a suitable \( \alpha \) that exhibits good continuity with respect to \( (\mu, m) \)\footnote{Here, continuity is understood in the weak sense: \( \alpha(\mu^n, m^n)(u) \) converges weakly to \( \alpha(\mu, m)(u) \) pointwise if \( (\mu^n(u), m^n(u)) \) converges weakly to \( (\mu(u), m(u)) \) pointwise.}, in order to ensure the applicability of the fixed point theorem.  According to Michael's Selection Theorem, a continuous selection is guaranteed to exist if the mapping has convex closed values and, crucially, is lower hemicontinuous. As the following example demonstrates, the best-response mapping $\mathbb{A}(\mu,m)$ fails to be lower hemicontinuous precisely at the points where the major player is indifferent between multiple actions.
\begin{example}
Consider a simple one-period ($T=1$) deterministic control problem for the major player with a binary action space $A = \{0, 1\}$ and no terminal reward. The objective is to choose a relaxed control  to maximize:
\[
\mathbb{E} \left[ \int_A f^0(a, m_0(U_0)) \, \alpha_0(U_0)(da) \right].
\]
For a fixed initial state $u \in W$, the optimal control $\alpha^*(u) \in \mathcal{P}(A)$ is any measure supported on $$\operatorname*{arg\,max}_{a \in A} f^0(a, m_0(u)).$$ This leads to:
\begin{itemize}
    \item If $f^0(1, m_0(u)) > f^0(0, m_0(u))$, the unique best response is $\alpha^*(u) = \delta_1$.
    \item If $f^0(1, m_0(u)) < f^0(0, m_0(u))$, the unique best response is $\alpha^*(u) = \delta_0$.
    \item If $f^0(1, m_0(u)) = f^0(0, m_0(u))$, the major player is indifferent, and any relaxed control $\alpha^*(u) = p\delta_1 + (1-p)\delta_0$ for $p \in [0,1]$ is optimal.
\end{itemize}
Now, consider a sequence of mean-field terms $\{m_0^n\}_n$ converging to $m_0$ such that $f^0(1, m_0^n(u)) > f^0(0, m_0^n(u))$ for all $n$, but in the limit, $f^0(1, m_0(u)) = f^0(0, m_0(u))$. For each $n$, the only possible selection is $\alpha^n(u) = \delta_1$. However, at the limit $m_0$, the set of best responses  is the entire simplex $\text{conv}\{\delta_0, \delta_1\}$. The sequence of unique best responses $\delta_1$ converges, but we could select a different limit point, such as $\delta_0$, which is also a valid best response at the limit. This demonstrates that the best-response mapping $\mathbb{A}$ is not lower hemicontinuous, and Michael's Selection Theorem does not apply.
\end{example}
In response to the aforementioned challenges using the direct arguments, we propose an auxiliary problem by adding an entropy regularization in the reward functional of the major player's problem. We first establish the existence of a regularized equilibrium for a simplified fixed-point mapping. Then, by letting the regularization parameter \( \lambda \) tend to 0, we show that the limit is a relaxed equilibrium of the original MFG problem. To this end, {
let us now introduce the Shannon entropy. Let $\mathcal{D}(A)$ be the set of probability distributions on the action space $A$ such that
\begin{itemize}
    \item If $A \subset \mathbb{R}^\ell$ with $\mathrm{Leb}(A) > 0$, $\mathcal{D}(A)$ is the set of probability density functions $\phi: A \to [0,\infty)$ with $\int_A \phi(a)da = 1$. The entropy is $\mathcal{H}(\phi) \coloneqq -\int_A \phi(a)\ln \phi(a)\,da$.
    \item If $A=\{0,1\}$, $\mathcal{D}(A)$ is the set of probability mass functions $\phi: A \to [0,1]$ with $\sum_{a \in A} \phi(a) = 1$. The entropy is $\mathcal{H}(\phi) \coloneqq -\sum_{a \in A} \phi(a)\ln \phi(a)$.
\end{itemize}
In both cases, the Shannon entropy $\mathcal{H}$ is a strictly concave functional on $\mathcal{D}(A)$.
\begin{definition}[Regularized  Controls]\label{def:regular_con}
A feedback relaxed control $\alpha \in \mathcal{A}$ is called a \emph{regularized (relaxed) control} if, for each  $(t,u) \in (\mathbb{T}\setminus\{T\}) \times W$,  the measure $\alpha_t(u)$ admits a density (or mass function) $\alpha_t(u)(\cdot) \in \mathcal{D}(A)$ such that $\mathcal{H}(\alpha_t(u)) > -\infty$.
We denote by $\mathcal{D}$ the set of all regularized controls.
\end{definition}
}
For a fixed regularization parameter $\lambda > 0$, we define the major player's $\lambda$-regularized reward functional as:
\begin{equation}\label{eq:major_regularized_reward}
J^0_{\lambda}(\alpha; \mu, m) \coloneqq J^0(\alpha; \mu, m) + \lambda\,\mathbb{E}\left[ \sum_{t=0}^{T-1} \mathcal{H}(\alpha_t(U_t)) \right].
\end{equation}
The major player's problem is now to maximize $J^0_{\lambda}$ over the set of regularized controls $\mathcal{D}$. Due to the strict concavity of the entropy term, for any given $(\mu, m)$, the regularized problem admits a unique optimal control, which we denote by $\alpha^\lambda(\mu, m)$. 
{\begin{definition}[Regularized Equilibrium]\label{def:regul}
A pair of mean-field flows \((\mu^*, m^*) \in \mathcal{C}_\mu \times \mathcal{C}_m\) is a \emph{$\lambda$-regularized equilibrium} if it is a fixed point of the mapping $\Phi_\lambda: \mathcal{C}_{\mu} \times \mathcal{C}_m \to 2^{\mathcal{C}_{\mu} \times \mathcal{C}_m}$ defined as follows:
\[
    \Phi_\lambda(\mu, m) \coloneqq \Gamma\big(\hat{\Theta}(\mu, m, \alpha^\lambda(\mu,m)), \alpha^\lambda(\mu,m), m\big).
\]
Here, $\alpha^\lambda(\mu,m)$ is the unique solution to the regularized problem \eqref{eq:major_regularized_reward}, and $\hat{\Theta}$ and $\Gamma$ are as defined in Definition \ref{def:relaxed}.
\end{definition}}

\begin{remark}
Both the relaxed and regularized equilibria are strong solutions in the sense of \cite{2014Carmona}, i.e., {in our case the mean-field terms are adapted to the major player's controlled state process}.
\end{remark}
\section{Technical Assumptions}\label{sect:technical assumptions}
\textbf{Topology.}
We equip all spaces of probability or sub-probability measures (e.g., $\mathcal{P}(S)$, $\mathcal{P}(A)$, $\mathcal{P}^{\text{sub}}(S \times W)$) with the \emph{topology of weak convergence}. All product spaces—namely, the space of relaxed controls $\mathcal{A}$, the spaces of mean-field flows $\mathcal{C}_m$ and $\mathcal{C}_\mu$, and the space of occupation measures $\mathcal{V}$—are endowed with the corresponding \emph{product topology}. A sequence $(m^n)$ converges to $m$ in $\mathcal{C}_m$ if $m_t^n(u) \to m_t(u)$ weakly for every pair $(t,u)$; convergence in $\mathcal{A}$, $\mathcal{C}_\mu$, and $\mathcal{V}$ is defined analogously on their respective components. Since the underlying state and action spaces $(S, W, A)$ are compact and the time horizon is finite, it follows from Prokhorov's and Tychonoff's theorems that all the aforementioned product spaces $\mathcal{A}$, $\mathcal{C}_m$, $\mathcal{C}_\mu$, and $\mathcal{V}$ are \emph{compact metric spaces}.

To facilitate our fixed point arguments and the convergence analysis of the vanishing entropy regularization, let us unpack the following standing assumptions imposed on the model. 
{\begin{assumption}\label{a1new:P^0}
\begin{enumerate}[\upshape (i)]
    \item \emph{(Major Player Transition Kernel):} For each fixed $(k,u) \in (\mathbb{T} \setminus \{T\}) \times W$, the stochastic kernel
    \(
        (a, m) \mapsto P^0_{k+1}(u, a, m, du)
    \)
    from $A \times \mathcal{P}^{\text{sub}}(S)$ to $\mathcal{P}(W)$ is continuous; that is, if $(a_n,m_n) \to (a,m)$ in $A \times \cP^{\text{sub}}(S)$, then \[ P^0_{k+1}(u,a_n,m_n,du) \;\rightarrow\; P^0_{k+1}(u,a,m,du) \quad \text{weakly}. \] Moreover, when $\operatorname{Leb}(A) > 0$, the mapping $a \mapsto P^0_{k+1}(u, a, m, u')$ is Lipschitz continuous, uniformly with respect to other arguments.

    \item \emph{(Minor Player Transition Kernel):} For each fixed $(k,u) \in (\mathbb{T} \setminus \{T\}) \times W$, the stochastic kernel
    \(
        (x, m) \mapsto P_{k+1}(x, u, m, dx)
    \)
    from $S \times \mathcal{P}^{\text{sub}}(S)$ to $\mathcal{P}(S)$ is continuous.
\end{enumerate}
\end{assumption}
\begin{remark}\label{remark:pi0}
	 Condition (i) in Assumption \ref{a1new:P^0} implies that for any
		$(k,u,u')\in \mathbb{T}\setminus\{T\}\times W^2$ the mapping $(\alpha,m) \mapsto \pi^0_{k+1}(u,u';\alpha,m)$ from $\mathcal{P}(A) \times \mathcal{P}^{\text{sub}}(S)$ to $[0,1]$, that is
		\begin{equation*}
			\pi^0_{k+1}(u,u';\alpha,m):=\int_{A}P^0_{k+1}(u,a,m,u')\alpha(da),
		\end{equation*}
	is continuous. Indeed, for any sequence $(\alpha_n,m_n)\to(\alpha,m)$ in $\mathcal{P}(A) \times \mathcal{P}^{\text{sub}}(S)$,
       {\small \begin{align*}
            \left|\pi^0_{k+1}(u,u';\alpha_n,m_n)-\pi^0_{k+1}(u,u';\alpha,m)\right|
            &\leq \left|\int_{A}P^0_{k+1}(u,a,m_n,u')\alpha_n(da)-\int_{A}P^0_{k+1}(u,a,m,u')\alpha_n(da)\right|\\
            &\quad +\left|\int_{A}P^0_{k+1}(u,a,m,u')\alpha_n(da)-\int_{A}P^0_{k+1}(u,a,m,u')\alpha(da)\right|\\
            &\leq \sup_{a\in A}\left|P^0_{k+1}(u,a,m_n,u')-P^0_{k+1}(u,a,m,u')\right|\\
            &\quad +\left|\int_{A}P^0_{k+1}(u,a,m,u')\alpha_n(da)-\int_{A}P^0_{k+1}(u,a,m,u')\alpha(da)\right|\to 0, 
        \end{align*}}as $n\to\infty$. The first term converges to zero because $P^0_{k+1}(u,\cdot,\cdot,u')$ is continuous on the compact set $A \times \mathcal{P}^{\text{sub}}(S)$ and thus uniformly continuous. The second term converges to zero by the definition of weak convergence of measures, as $a \mapsto P^0_{k+1}(u,a,m,u')$ is a bounded and continuous function on $A$.
\end{remark}}
\begin{remark}\label{remark:pi}
	{By Assumption \ref{a1new:P^0} and Remark \ref{remark:pi0}, we conclude that  
		$\pi_t(x, u; \alpha, m; dx', du') : S \times W \times \mathcal{P}(A) \times \mathcal{P}^{\text{sub}}(S) \to \mathcal{P}(S \times W)$,  
		where  
		\[
		\pi_t(x, u; \alpha, m; dx', du') := P_{t}(x, u, m, dx') \pi^0_{t}(u, du'; \alpha, m),
		\]  
		is continuous for each \( t \in \mathbb{T}^* \). }
\end{remark}
{\begin{assumption}\label{a2new:f0g0}
\begin{enumerate}[\upshape (i)]
    \item \emph{(Major Player Running Reward):} For each fixed $(t,x) \in (\mathbb{T} \setminus \{T\}) \times S^0$, the running reward function
    \(
        (a, m) \mapsto f^0_t(x, a, m)
    \)
    is continuous. Moreover, when $\operatorname{Leb}(A) > 0$, the mapping $a \mapsto f^0_t(x, a, m)$ is Lipschitz continuous, uniformly with respect to other arguments.

    \item \emph{(Major Player Terminal Reward):} For each fixed $x \in S^0$, the terminal reward function
    \(
        m \mapsto g^0(x, m)
    \)
    is continuous.
\end{enumerate}
\end{assumption}
\begin{assumption}\label{a3:minor}
\begin{enumerate}[\upshape (i)]
    \item \emph{(Minor Player Continuation  Reward):} For each fixed $(t,u) \in (\mathbb{T} \setminus \{T\}) \times W$, the continuation reward function
    \(
        (x, m) \mapsto f_t(x, u, m)
    \)
    is continuous.

    \item \emph{(Minor Player Stopping  Reward):} For each fixed $(t,u) \in \mathbb{T} \times W$, the stopping reward function
    \(
        (x, m) \mapsto g_t(x, u, m)
    \)
    is continuous.
\end{enumerate}
\end{assumption}}
We next state three auxiliary lemmas that will be used later in the analysis. Their proofs are standard and are therefore deferred to Appendix \ref{appendix:proof}.
\begin{lemma}\label{lemma:r[m,a]}
Under Assumption \ref{a1new:P^0}, the set $\mathcal{R}[m;\alpha]$ defined in Definition \ref{constraint} is compact.
\end{lemma}
\begin{lemma}\label{lemma:Gamma}
	 Let Assumptions \ref{a1new:P^0} and \ref{a3:minor} hold. Then, the map $$\mathcal{V}\times(\mathcal{C}_{\mu}\times\mathcal{C}_{m}) \ni((\bar{\mu}, \bar{m}),(\mu, m)) \mapsto J^1(\bar{\mu}, \bar{m};\mu, m) \in \mathbb{R}$$ defined in Definition \ref{def:lpop} is continuous. 
\end{lemma} 
{\begin{lemma}\label{lemma:conditioning}
    Let Assumption \ref{a1new:P^0} hold. The conditioning mapping $\Gamma:\mathcal{V}\times\mathcal{A}\times\mathcal{C}_m\rightarrow 2^{\mathcal{C}_{\mu}\times\mathcal{C}_{m}}$ defined in Definition \ref{def:relaxed} has a closed graph.
\end{lemma}}
\begin{remark}\label{remark:convex}
   Lemmas~\ref{lemma:r[m,a]}--\ref{lemma:Gamma} ensure that, for any triple $(\mu,m,\alpha)\in\mathcal{C}_\mu\times\mathcal{C}_m\times\mathcal{A}$,   $\hat{\Theta}(\mu,m,\alpha)$ defined in Definition \ref{def:relaxed} is nonempty. In combination with Lemma~\ref{lem:properties_gamma_theta}, this further implies that $\Gamma(\hat{\Theta}(\mu,m,\alpha),\alpha,m)$ constitutes a nonempty convex subset of $\mathcal{C}_\mu \times \mathcal{C}_m$.
\end{remark}

 For the case of $A\subset \mR^\ell$  with \( \text{Leb}(A) > 0 \), we require the following assumption from  \cite{erhan2024} to properly control the entropy term.
\begin{assumption}\label{a:A}
	When $\ell>1$, there exists $\vartheta>0$ and $\iota \in(0, \pi / 2]$ such that for any $a \in A$, there is a cone with vertex $a$ and angle $\iota($ denoted by $\operatorname{cone}(a, \iota))$ that satisfies $\left(\operatorname{cone}(a, \iota) \cap B_{\vartheta}(a)\right) \subseteq A$. When $\ell=1$, there exists $\vartheta>0$ such that for any $a \in A$, either $[a-\vartheta, a]$ or $[a, a+\vartheta]$ is contained in $A$.
\end{assumption}

\section{Existence of Regularized Equilibrium}\label{sect:regul}
Given \( (\mu, m) \in \mathcal{C}_{\mu} \times \mathcal{C}_m \), we will solve the regularized control problem \eqref{eq:major_regularized_reward} using the dynamic programming principle (DPP). Define the value function by
\begin{equation*}
\begin{split}
    	V^{\lambda}(t,u)=&\sup_{\tilde{\alpha}\in\mathcal{D}_{t}}\mathbb{E} \bigg[ \sum_{s=t}^{T-1} \left( \int_{A} f_{s}^0(\Psi_s(U_s), a, m_s(U_s)) \tilde{\alpha}_s(U_s)(da) + \lambda \mathcal{H}(\tilde{\alpha}_s(U_s)) \right) \\&\qquad+ g^0(\Psi_T(U_T), \mu_{T}(U_{T}))\bigg|U_{t}=u \bigg],
\end{split}
\end{equation*}
where $\mathcal{D}_t$ denotes the set of admissible regularized controls from time $t$ onwards.
By DPP, we can derive that
\begin{equation*}
	\begin{split}
		&V^{\lambda}(t,u)=\max_{\tilde{\alpha}\in {\mathcal{D}(A)}}\bigg\{\int_{A} f_{t}^0(\Psi_t(u), a, m_t(u)) \tilde{\alpha}(u)(da) + \lambda \mathcal{H}(\tilde{\alpha}(u)) \\&\qquad\qquad\qquad\qquad+\sum_{u'\in W}V^{\lambda}(t+1,u')\int_{A}P^0_{t+1}(u,a,m_{t}(u),u')\tilde{\alpha}(u)(da)\bigg\},\\
		&V^{\lambda}(T,u)=g^0(\Psi_T(u), \mu_{T}(u)).
	\end{split}
\end{equation*}
{It follows from the strict concavity of the entropy  and the first-order condition that the optimal regularized control is characterized by}
\begin{equation*}
	\alpha^{\lambda}_t(u)(a)=\frac{e^{\frac{1}{\lambda}(f^0_{t}(\Psi_t(u), a, m_t(u))+\sum_{u'\in W}V^{\lambda}(t+1,u')P^0_{t+1}(u,a,m_{t}(u),u'))}}{\int_Ae^{\frac{1}{\lambda}(f^0_{t}(\Psi_t(u), a, m_t(u))+\sum_{u'\in W}V^{\lambda}(t+1,u')P^0_{t+1}(u,a,m_{t}(u),u'))}da},\quad t\in \mathbb{T}\setminus\{T\}.
\end{equation*}
For the case that $A=\left\{0,1\right\}$, we have
\begin{align*}
&\alpha^{\lambda}_t(u)(1)=\frac{e^{\frac{1}{\lambda}(f^0_{t}(\Psi_t(u), 1, m_t(u))+\sum_{u'\in W}V^{\lambda}(t+1,u')P^0_{t+1}(u,1,m_{t}(u),u'))}}{\sum_{i=0}^1e^{\frac{1}{\lambda}(f^0_{t}(\Psi_t(u), i, m_t(u))+\sum_{u'\in W}V^{\lambda}(t+1,u')P^0_{t+1}(u,i,m_{t}(u),u'))}},\quad t\in \mathbb{T}\setminus\{T\},\\
&\alpha^{\lambda}_t(u)(0)=\frac{e^{\frac{1}{\lambda}(f^0_{t}(\Psi_t(u), 0, m_t(u))+\sum_{u'\in W}V^{\lambda}(t+1,u')P^0_{t+1}(u,0,m_{t}(u),u'))}}{\sum_{i=0}^1e^{\frac{1}{\lambda}(f^0_{t}(\Psi_t(u), i, m_t(u))+\sum_{u'\in W}V^{\lambda}(t+1,u')P^0_{t+1}(u,i,m_{t}(u),u'))}},\quad t\in \mathbb{T}\setminus\{T\}.	
\end{align*}
\begin{remark}\label{remark:alphacontinu}
{By Assumptions \ref{a1new:P^0} and \ref{a2new:f0g0}, and using backward induction, we conclude that \(\alpha^{\lambda}\) is continuous in \((\mu,m)\). In the sequel, when we need to emphasize the dependence  \(\alpha^{\lambda}\) on \((\mu,m)\), we write \(\alpha^\lambda\) as \(\alpha^\lambda(\mu,m)\) or $(\alpha_k^\lambda(u;\mu,m))_{k,u}$.}
\end{remark}

To establish the existence of a regularized equilibrium, we first demonstrate the continuity of the set $\mathcal{R}^*(\mu,m):=\mathcal{R}[m;\alpha^\lambda(\mu,m)]$, in the context of set-valued mappings. Afterwards, we prove the existence of a fixed point for $\Phi_\lambda$ using Kakutani-Fan-Glicksberg's fixed point theorem.
\begin{lemma}\label{setR*contin}
Under Assumptions \ref{a1new:P^0}-\ref{a3:minor}, the set-valued mapping $\mathcal{R}^*:(\mu,m) \mapsto \mathcal{R}[m; \alpha^\lambda(\mu,m)]$ is continuous from $\mathcal{C}_\mu \times \mathcal{C}_m$ to $2^\mathcal{V}$.
\end{lemma}
\begin{proof}
	\textit{Step 1}. We first prove the upper hemicontinuity. By the Closed Graph Theorem, it suffices to show that $\mathcal{R}^*$ has closed graph. Consider a sequence $\left(\tilde{\mu}^n, \tilde{m}^n\right) \in \mathcal{R}^*(\overline{\mu}^n,\overline{m}^n)=\mathcal{R}[\overline{m}^n;\alpha^{\lambda}(\overline{\mu}^n,\overline{m}^n)]$ such that $\left(\tilde{\mu}^n, \tilde{m}^n\right)_{n \geq 1}\rightharpoonup(\tilde{\mu}, \tilde{m})$  and $(\overline{\mu}^n,\overline{m}^n)\rightharpoonup(\overline{\mu},\overline{m})$. Let us show that $(\tilde{\mu}, \tilde{m}) \in \mathcal{R}^*(\overline{\mu},\overline{m})$.  

Since $\left(\tilde{\mu}^n, \tilde{m}^n\right) \in \mathcal{R}^*(\overline{\mu}^n, \overline{m}^n)$, for all $\varphi \in C_b(\mathbb{T} \times S \times W)$, it holds that
		\begin{align*}
			\sum_{t=0}^{T} \int_{S\times W} \varphi(t, x, u) \tilde{\mu}^n_t(dx,du)
			&=  \int_{S\times W} \varphi(0, x, u) m_0^*(dx,du)\\&+\sum_{t=0}^{T-1} \int_{S\times W} \mathcal{L}(\varphi)(t, x, u;\alpha^{\lambda}_t(u;\overline{\mu}^n,\overline{m}^n),\overline{m}^n_{t}(u))\tilde{m}^n_t(dx,du).
		\end{align*}
        We need to show that this equality holds in the limit. By weak convergence, $\int \varphi \, d\tilde{\mu}^n_t \to \int \varphi \, d\tilde{\mu}_t$. The critical step is to show the convergence of the integral involving the generator.
        Note that
	{\small	\begin{align*}
			\int_{S\times W}& \mathcal{L}(\varphi)(t, x, u;\alpha^{\lambda}_t(u;\overline{\mu}^n,\overline{m}^n),\overline{m}^n_{t}(u))\tilde{m}^n_t(dx,du)\\&=\sum_{u',u\in W}\pi^0_{t+1}(u,u';\alpha^{\lambda}_t(u;\overline{\mu}^n,\overline{m}^n),\overline{m}^n_{t}(u))\int_{S}\bigg(\int_{S}\varphi(t+1,x',u')P_{t+1}(x,u,\overline{m}^n_{t}(u),dx')\bigg)\tilde{m}^n_t(dx,\{u\})\\&-\int_{S\times W}\varphi(t,x,u)\tilde{m}^n_t(dx,du).
		\end{align*}}
	  { Then, combining Assumption \ref{a1new:P^0}, Remark \ref{remark:pi0}, and Remark \ref{remark:alphacontinu}, and using the same argument as in Remark \ref{remark:pi0},  we have
	{\small	$$
	\lim _{n \rightarrow \infty} 	\int_{S\times W} \mathcal{L}(\varphi)(t, x, u;\alpha^{\lambda}_t(u;\overline{\mu}^n,\overline{m}^n),\overline{m}^n_{t}(u))\tilde{m}^n_t(dx,du)= \int_{S\times W} \mathcal{L}(\varphi)(t, x, u;\alpha^{\lambda}_t(u;\overline{\mu},\overline{m}),\overline{m}_{t}(u))\tilde{m}_t(dx,du) .
		$$}}
		It readily follows that \( (\tilde{\mu}, \tilde{m}) \in \mathcal{R}^*(\overline{\mu}, \overline{m}) \).

	\textit{Step 2.} We now prove the lower hemicontinuity. Consider a sequence \( (\overline{\mu}^n, \overline{m}^n)_{n \geq 1} \subset \mathcal{C}_{\mu}\times\mathcal{C}_{m} \) such that \( (\overline{\mu}^n, \overline{m}^n) \rightharpoonup (\overline{\mu}, \overline{m}) \) and let \( (\mu, m) \in \mathcal{R}^{\star}(\overline{\mu}, \overline{m}) = \mathcal{R}[\bar{m};\alpha^{\lambda}(\overline{\mu},\overline{m})] \). Our task is to show that, up to a subsequence, we can find  \( (\mu^n, m^n) \in \mathcal{R}^{\star}(\overline{\mu}^n, \overline{m}^n) = \mathcal{R}[\bar{m}^n;\alpha^{\lambda}(\overline{\mu}^n,\overline{m}^n)] \) and \( (\mu^n, m^n) \rightharpoonup (\mu, m) \). 
	
	
By the argument in Appendix \ref{appendixB}, we can construct a filtered probability space \((\Omega, \mathcal{F}, \mathbb{F}, P)\) such that \((\tilde{X}^n, \tilde{U}^n)\) and \((\tilde{X}, \tilde{U})\) are \(\mathbb{F}\)-Markov chains taking values in \(S \times W\), with transition kernels  
\[
\pi(x, u; \alpha^{\lambda}(\overline{\mu}^n,\overline{m}^n), \overline{m}^n; dx', du') \quad \text{and} \quad \pi(x, u; \alpha^{\lambda}(\overline{\mu},\overline{m}), \overline{m}; dx', du'),
\]  
and the same initial distribution \(m_0^*\). Moreover,  
\[
(\tilde{X}^n_k, \tilde{U}^n_k) \to (\tilde{X}_k, \tilde{U}_k), \quad \forall k \in\mathbb{T}, \, P\text{-a.s.}
\]
		
Following the argument in Theorem~18 of Appendix~A in \cite{Peter2023}, we conclude that for \( (\mu, m) \in \mathcal{R}[\bar{m}; \alpha^{\lambda}(\overline{\mu},\overline{m})] \), it can be associated with randomized stopping times. Specifically, we can extend the filtered probability space \((\Omega, \mathcal{F}, \mathbb{F}, P)\) to \((\overline{\Omega}, \overline{\mathcal{F}}, \overline{\mathbb{F}}, \overline{\mathbb{P}})\) such that
 \begin{itemize}
 	\item  \((\overline{\Omega}, \overline{\mathcal{F}}, \overline{\mathbb{P}})\) is a complete probability space endowed with a complete filtration \(\overline{\mathbb{F}}\), which supports random variables $(\overline{\tau},\overline{X}, \overline{U},\overline{X}^1, \overline{U}^1,\cdots)$.
 	
 	\item \((\overline{X}, \overline{U})\) and \((\overline{X}^n,\overline{U}^n)\) are \(\overline{\mathbb{F}}\)-Markov chains valued in \(S \times W\), with transition kernels  \[ \pi(x, u; \alpha^{\lambda}(\overline{\mu},\overline{m}), \overline{m}; dx', du') \quad \text{and} \quad \pi(x, u; \alpha^{\lambda}(\overline{\mu}^n,\overline{m}^n), \overline{m}^n; dx', du'),\]  
     and the same initial distribution \(m_0^*\), satisfying  \(( \overline{X}^n_k,\overline{U}^n_k) \to (\overline{X}_k, \overline{U}_k)\),  \(\forall k \in \mathbb{T}\), \(\overline{\mP}\)-a.s.

 	\item $\bar{\tau}$ is an $\overline{\mathbb{F}}$-stopping time valued in $\mathbb{T}$.

 	\item The measures have the following representations:
 	\begin{align*}
 		& m_t(B\times \{u\})=\overline{\mathbb{P}}\left[\overline{X}_t \in B, \,\overline{U}_t=u,\, t<\overline{\tau}\right], \quad B \in \mathcal{B}(S), \quad u \in W,\quad t \in \mathbb{T} \backslash\{T\}, \\
 		& \mu_t(B\times \{u\})=\overline{\mathbb{P}}\left[\overline{X}_t \in B,\, \overline{U}_t=u,\, \overline{\tau}=t \right], \quad B \in \mathcal{B}(S), \quad u \in W, \quad t \in \mathbb{T}.
 	\end{align*}
 \end{itemize}
  Now, let us define
\begin{align*}
	& m^n_t(B\times \{u\}):=\overline{\mathbb{P}}\left[\overline{X}^n_t \in B,\,\overline{U}^n_t=u,\, t<\overline{\tau}\right], \quad B \in \mathcal{B}(S), \quad u \in W, \quad t \in \mathbb{T} \backslash\{T\},  \\
	& \mu^n_t(B\times \{u\}):=\overline{\mathbb{P}}\left[\overline{X}^n_t \in B, \, \overline{U}^n_t=u,\, \overline{\tau}=t \right], \quad B \in \mathcal{B}(S), \quad u \in W,\quad t \in \mathbb{T}.
\end{align*}
 Then, we have that \( (\mu^n, m^n) \in \mathcal{R}^{\star}(\bar{\mu}^n, \bar{m}^n) = \mathcal{R}[\bar{m}^n;\alpha^{\lambda}(\overline{\mu}^n,\overline{m}^n)] \) by an argument similar to that  in subsection \ref{Derivation of the constraint}.

We next proceed to show that $(\mu^n, m^n) \rightharpoonup (\mu, m)$. For any \(\varphi\in C_{b}( S\times W) \), we have
\begin{align*}
\int_{S\times W}\varphi(x,u)m^n_t(dx,du)-\int_{S\times W}\varphi(x,u)m_t(dx,du)&=\mE^{\overline{\mP}}[\varphi(\overline{X}^n_t,\overline{U}^n_t)\mathbf{1}_{t<\bar{\tau}}]-\mE^{\overline{\mP}}[\varphi(\overline{X}_t,\overline{U}_t)\mathbf{1}_{t<\bar{\tau}}].
\end{align*}
  By the a.s. convergence of $(\overline{X}^n_t, \overline{U}^n_t)$ to $(\overline{X}_t, \overline{U}_t)$ and the Bounded Convergence Theorem, the right-hand side converges to zero. A similar argument holds for $\mu^n$. Therefore,  $(\mu^n, m^n) \rightharpoonup (\mu, m)$, which completes the proof of lower hemicontinuity.
\end{proof}

By a direct modification of the proof of Lemma \ref{setR*contin}, we can get the next result.

\begin{lemma}\label{sethatr_con}
Suppose that Assumptions \ref{a1new:P^0}and \ref{a3:minor}  hold. Then,  the set-valued mapping $(m, \alpha) \mapsto \mathcal{R}[m; \alpha]$ is continuous from $\mathcal{C}_m \times \mathcal{A}$ to $2^\mathcal{V}$.
\end{lemma}

\begin{theorem}\label{them:regular_existence}
Suppose that Assumptions \ref{a1new:P^0}-\ref{a3:minor} are satisfied, for every $\lambda>0$, there exists a regularized equilibrium.
\end{theorem}
\begin{proof}
	 The proof consists of applying the Kakutani-Fan-Glicksberg fixed-point theorem to the set-valued mapping $\Phi_\lambda: \mathcal{C}_\mu \times \mathcal{C}_m \to 2^{\mathcal{C}_\mu \times \mathcal{C}_m}$, defined as:
\[
    \Phi_\lambda(\mu, m) = \Gamma\big(\hat{\Theta}(\mu, m, \alpha^\lambda(\mu,m)), \alpha^\lambda(\mu,m), m\big).
\] First, note that $\mathcal{C}_{\mu}\times\mathcal{C}_{m}$ is nonempty, compact, convex and is included in
	$$
	\prod_{t \in \mathbb{T}} \mathcal{M}^s(S)^{W} \times \prod_{t \in \mathbb{T} \backslash\{T\}} \mathcal{M}^s(S)^{W},
	$$
	which is a locally convex Hausdorff space. Moreover,  Remark \ref{remark:convex} implies that  $\Phi_\lambda$  has nonempty convex values. 
   We now show that $\Phi_\lambda$ has a closed graph. Let $((\mu^n, m^n), (\mu'^n, m'^n))$ be a sequence in the graph of $\Phi_\lambda$ converging to $((\mu, m), (\mu', m'))$. We must show that $(\mu', m') \in \Phi_\lambda(\mu, m)$. The proof proceeds by leveraging the continuity of the intermediate mappings. Let $\alpha^\lambda_n \coloneqq \alpha^\lambda(\mu^n, m^n)$ and let $(\tilde{\mu}^n, \tilde{m}^n) \in \hat{\Theta}(\mu^n, m^n, \alpha^\lambda_n)$ be a sequence such that $(\mu'^n, m'^n) \in \Gamma((\tilde{\mu}^n, \tilde{m}^n), \alpha^\lambda_n, m^n)$.
\begin{enumerate}
    \item By Remark \ref{remark:alphacontinu}, $\alpha^\lambda_n \to \alpha^\lambda(\mu, m)$.
    \item The mapping $(\mu,m) \mapsto \hat{\Theta}(\mu, m, \alpha^\lambda(\mu,m))$ is upper hemicontinuous and has nonempty compact values by Berge's Maximum Theorem (e.g., Theorem 17.31 in \cite{aliprantis06}), for which Lemmas \ref{lemma:Gamma} and \ref{setR*contin} are required. This implies that the sequence $(\tilde{\mu}^n, \tilde{m}^n)$ has a convergent subsequence whose limit $(\tilde{\mu}, \tilde{m})$  lies in $\hat{\Theta}(\mu, m, \alpha^\lambda(\mu,m))$ (cf. Theorem 17.16 in \cite{aliprantis06}).
    \item By Lemma \ref{lemma:conditioning}, the conditioning mapping $\Gamma$ has a closed graph, which implies that $$(\mu', m') \in \Gamma((\tilde{\mu}, \tilde{m}), \alpha^\lambda(\mu, m), m).$$
\end{enumerate}
This establishes that $(\mu', m') \in \Phi_\lambda(\mu, m)$, so the graph is closed.
    With the help of Kakutani-Fan-Glicksberg’s fixed point theorem, we conclude that a regularized equilibrium exists.
\end{proof}

\section{Existence of Relaxed Equilibrium}\label{sect:relaxed}

Back to our original MFG problem, we need to address the existence of the relaxed equilibrium in Definition \ref{def:relaxed}. As a direct modification of Lemma 1 in \cite{erhan2024}, we first have the next result.

\begin{lemma}\label{lemma:H}
Let Assumptions \ref{a1new:P^0}, \ref{a2new:f0g0}, and \ref{a:A} hold. Then, for sufficiently small \( \lambda > 0 \) and for any $t\in \mathbb{T}\setminus\{T\}$,
\[
\sup_{u \in W} |\mathcal{H}(\alpha^{\lambda}_t(u))| \leq \kappa_{1} + \kappa_2 |\ln\lambda| + \ell \ln(1 + \sup_{u \in W}|V^{\lambda}(t+1,u)|),
\]
where \( \kappa_{1} \) and \( \kappa_{2} \) are positive constants independent of \( \lambda \).
\end{lemma}
Then, by the dynamic programming equation, we obtain
\[
\sup_{u \in W} |V^{\lambda}(t, u)| \leq M + \kappa_1 \lambda + \kappa_2 \lambda |\ln \lambda| + \ell \lambda \ln(1 + \sup_{u \in W} |V^{\lambda}(t+1, u)|) + \sup_{u \in W} |V^{\lambda}(t+1, u)|,
\]
where $M$ is a positive function independent of \( \lambda \).
Thus, by induction, we arrive at the following result:
\begin{lemma}\label{lemma:Vlambda}
Let Assumptions \ref{a1new:P^0}, \ref{a2new:f0g0}, and \ref{a:A} hold.	Given \( \lambda_n \in (0, 1] \) with \( \lambda_n \downarrow 0 \), we have \[\sup_{n,t,u} |V^{\lambda_n}(t,u)| < \infty .\]
\end{lemma}
We are now ready to present the main result.
\begin{theorem}\label{thm:existencerelaxed}
Under Assumptions \ref{a1new:P^0}-\ref{a3:minor}, we have the following results:
\begin{itemize}
    \item If  $A\subset \mR^\ell$  with \( \text{Leb}(A) > 0 \) and Assumption \ref{a:A} holds, then a relaxed equilibrium exists.
    \item If \( A = \{0, 1\} \), no further assumption is needed, and a relaxed equilibrium exists.
\end{itemize}
\end{theorem}
\begin{proof}
\textbf{Step 1: Major player's problem.} 

\textit{Case 1: $A\subset \mR^\ell$  with \( \text{Leb}(A) > 0 \).}
Since \( a \mapsto f^0_{t}(x, a, m) \) and \( a \mapsto P_{t+1}^0(u, a, m, u') \) are continuous (by Assumptions \ref{a1new:P^0} and \ref{a2new:f0g0} ), and \( A \) is compact, for any \( (t, y, u,m) \in \mathbb{T} \setminus \{T\} \times \mathbb{R}^{|W|} \times W\times\cP^{sub}(S)^{W} \), we have that
\begin{equation}\label{E}
    E_t(y, u, m) := \operatorname{arg \, max}_{a \in A} \left\{ f^0_{t}(\Psi_t(u), a, m(u)) + P^0_{t+1}(u, a, m(u), \cdot) \cdot y \right\}
\end{equation}
is nonempty and compact, where $P^0_{t+1}(u, a, m(u), \cdot):=(P^0_{t+1}(u, a, m(u), u'))_{u'\in W}$. For any \( (t,y,m) \in \mathbb{T} \setminus \{T\}\times \mathbb{R}^{|W|}\times\cP^{sub}(S)^{W} \), consider the collection
\[
\Upsilon_t(y,m) := \{ \alpha \in \cP(A)^{W} : \operatorname{supp}(\alpha(u)) \subset E_t(y, u, m), \quad \forall u \in W \}.
\]

Let \( \{\lambda_n\}_{n \geq 1} \) be a sequence in \( (0, 1] \) with \(\lambda_n \downarrow 0 \). For each \(n\), Theorem~\ref{them:regular_existence} guarantees the existence of a $\lambda_n$-regularized equilibrium. This provides a sequence of tuples \( \{(\mu^n, m^n, \alpha^n, (\tilde{\mu}^n, \tilde{m}^n))\}_{n \geq 1} \), where $\alpha^n = \alpha^{\lambda_n}(\mu^n, m^n)$ is the major player's regularized best response, $(\tilde{\mu}^n, \tilde{m}^n) \in \hat{\Theta}(\mu^n, m^n, \alpha^n)$, and the consistency condition $(\mu^n, m^n) \in \Gamma((\tilde{\mu}^n, \tilde{m}^n), \alpha^n, m^n)$ holds. Since  $\mathcal{C}_{\mu} \times \mathcal{C}_{m}\times\mathcal{A}\times\mathcal{V}$ is compact, the sequence of equilibrium tuples admits a convergent subsequence, which we do not relabel, converging to a limit point \( (\mu, m, \alpha, (\tilde{\mu}, \tilde{m})) \). 

{ Furthermore, because the sequence 
\(\{V^{\lambda_n}(t, u)\}_{n \in \mathbb{N}}\) is bounded in \(\mathbb{R}\) by Lemma \ref{lemma:Vlambda} 
and the time and state spaces \(\mathbb{T}\) and \(W\) are finite, there exists a subsequence, still denoted by \(\{V^{\lambda_n}(t, u)\}_{n \in \mathbb{N}}\), that converges for any \((t,u)\in \mathbb{T}\times W\). We denote the limit by \(V(t,u) \in \mathbb{R}\).}


Now, we claim that \( \alpha_t \in \Upsilon_t(V(t+1,\cdot),m_t) \) for every $t\in\mathbb{T}\setminus\{T\}$, that is, we need to show that \( \alpha_t(u) \in \mathcal{P}(A) \) is supported by the compact set \( E_t(V(t+1,\cdot), u,m_t) \subset A \)  for all \((t, u) \in\mathbb{T}\setminus \{T\} \times W\). Here, $V(t+1,\cdot):=(V(t+1,u'))_{u'\in W}$. Now let us fix a $t\in \mathbb{T} \setminus \{T\}$. Because this holds trivially when \( E_t(V(t+1,\cdot), u,m_t) = A \), we assume \( E_t(V(t+1,\cdot), u,m_t) \subsetneq A\) in the following. Our goal is to prove that for any \( u \in W \), \( \alpha_t(u)(B_r(a_0)) = 0 \) for all \( a_0 \in A \) and \( r > 0 \) such that
\[
\overline{B_r(a_0)} \cap E_t(V(t+1,\cdot), u,m_t) = \emptyset \quad \text{and} \quad B_r(a_0) \subseteq A.
\]
Note that this readily implies \( \operatorname{supp}(\alpha_t(u)) \subset E_t(V(t+1,\cdot), u,m_t) \) for all \( u\in W \), as desired. To this end, take a continuous and bounded function \( h : A \to [0,1] \) such that
\begin{equation}\label{aux:h}
h(a) \equiv 1 \quad \text{for } u \in B_r(a_0) \quad \text{and} \quad h(a) \equiv 0 \quad \text{for } u \notin \overline{B_{r + d/2}(a_0)},
\end{equation}
where \( d := \operatorname{dist}(\overline{B_r(a_0)}, E_t(V(t+1,\cdot), u,m_t)) \) is strictly positive because \( \overline{B_r(a_0)} \) and \( E_t(V(t+1,\cdot), u,m_t) \) are disjoint compact sets. 

Consider \( \beta_t(u) := \max_{a \in A} \{ f_t^0(\Psi_t(u), a, m_{t}(u)) + P^0_{t+1}(u, a, m_t(u), \cdot) \cdot V(t+1,\cdot) \} < \infty \). Note that \( d > 0 \) implies
\[
\varepsilon := \beta_t(u) - \sup \{ f_t^0(\Psi_t(u), a, m_{t}(u)) + P^0_{t+1}(u, a, m_t(u), \cdot) \cdot V(t+1,\cdot): a \in \overline{B_r(a_0)} \} > 0.
\]
 { Also,  the continuity of \( a \mapsto  f_t^0(\Psi_t(u), a, m_{t}(u)) + P^0_{t+1}(u, a, m_t(u), \cdot) \cdot V(t+1,\cdot) \)  ensures that the set of near-maximizers contains an open neighborhood around the maximizer. Moreover, the geometric property of the set~$A$ (Assumption~\ref{a:A}) guarantees this neighborhood has positive volume within~$A$.} It thus follows that
\[
L := \operatorname{Leb} \left( \left\{ a \in A : f_t^0(\Psi_t(u), a, m_{t}(u)) + P^0_{t+1}(u, a, m_t(u), \cdot) \cdot V(t+1,\cdot) {\geq} \beta_t(u) - \frac{\varepsilon}{4} \right\} \right) > 0.
\]
{As \( V^{\lambda_n}(t+1,u) \to V(t+1,u) \), $f_t^0(\Psi_t(u), a, m^n_{t}(u))\rightarrow f_t^0(\Psi_t(u), a, m_{t}(u))$ and $P^0_{t+1}(u, a, m^n_t(u), u')\rightarrow P^0_{t+1}(u, a, m_t(u), u')$ for all \( n \in \mathbb{N} \) large enough}, we have:
\begin{align*}
	&\beta_t(u)- \sup \{ f_t^0(\Psi_t(u), a, m^n_{t}(u)) + P^0_{t+1}(u, a, m^n_t(u), \cdot) \cdot V^n(t+1,\cdot): a \in \overline{B_r(a_0)} \} > \frac{\varepsilon}{2}, \\ 
	&\operatorname{Leb} \left( \left\{ a \in A : f_t^0(\Psi_t(u), a, m^n_{t}(u)) + P^0_{t+1}(u, a, m^n_t(u), \cdot) \cdot V^n(t+1,\cdot) {\geq} \beta_t(u) - \frac{\varepsilon}{2} \right\} \right) \geq \frac{L}{2},
\end{align*}
where the first inequality follows from the Berge Maximum Theorem.
As a result, we have that
\begin{align*}
	&\limsup_{n \to \infty} \int_A h(a) \alpha^{n}_t(u)(da)\\ &\leq \limsup_{n \to \infty} \int_{\overline{B_{r + d/2}(a_0)}} h(a) \frac{e^{\frac{1}{\lambda_n}(f^0_{t}(\Psi_t(u), a, m^n_t(u))+P^0_{t+1}(u, a, m^n_t(u), \cdot) \cdot V^n(t+1,\cdot))}}{\int_Ae^{\frac{1}{\lambda_n}(f^0_{t}(\Psi_t(u), a, m^n_t(u))+P^0_{t+1}(u, a, m^n_t(u), \cdot) \cdot V^n(t+1,\cdot))}(da)}(da)\\
	&= \limsup_{n \to \infty} \int_{\overline{B_{r + d/2}(a_0)}} h(a) \frac{e^{\frac{1}{\lambda_n}(f^0_{t}(\Psi_t(u), a, m^n_t(u))+P^0_{t+1}(u, a, m^n_t(u), \cdot) \cdot V^n(t+1,\cdot) - (\beta_t(u) - \epsilon / 2))}}{\int_A e^{\frac{1}{\lambda_n}(f^0_{t}(\Psi_t(u), a, m^n_t(u))+P^0_{t+1}(u, a, m^n_t(u), \cdot) \cdot V^n(t+1,\cdot) - (\beta_t(u) - \epsilon / 2))}(da)}(da)\\
	&\leq \limsup_{n \to \infty} \int_{\overline{B_{r + d/2}(a_0)}} h(a) \frac{e^{-\frac{\epsilon}{2\lambda_n}}}{\int_{\left\{ a \in A : f_t^0(\Psi_t(u), a, m^n_{t}(u)) + P^0_{t+1}(u, a, m^n_t(u), \cdot) \cdot V^n(t+1,\cdot) \geq \beta_t(u) - \frac{\varepsilon}{2} \right\}}1(da)}(da)                                         \\
	&\leq \limsup_{n \to \infty} \frac{2e^{-\frac{\epsilon}{2 \lambda_n}}}{L} \operatorname{Leb}\left(\overline{B_{r + d/2}(a_0)}\right) = 0. 
\end{align*}
In view that \( h \) is continuous and bounded that satisfies \eqref{aux:h}, it holds that
\begin{align*}
	&\alpha_t(u)\left(\overline{B_r(a_0)}\right) \leq \int_A h(a) \alpha_t(u)(da) = \lim_{n \to \infty} \int_A h(a) \alpha^{n}_t(u)(da) = 0
\end{align*}
The claim \( \alpha_t \in \Upsilon_t(V(t+1, \cdot), m_t) \) for every \( t \in \mathbb{T} \setminus \{T\} \) follows from the arbitrariness of \( t \).

We next prove that for any \((t,u)\in \mathbb{T}\setminus\{T\} \times W\),
\begin{align*}
	&V(t,u) = \int_{A} f_{t}^0(\Psi_t(u), a, m_t(u)) \alpha_t(u)(da) + \sum_{u' \in W} V(t+1, u') \int_{A} P^0_{t+1}(u, a, m_t(u), u') \alpha_t(u)(da),\\
	&V(T,u) = g^0(\Psi_T(u),\mu_T(u)),
\end{align*}
which, combined with the fact that \( \alpha_t \in\Upsilon_t(V(t+1, \cdot), m_t) \) for every \( t \in \mathbb{T} \setminus \{T\} \), implies that \( V \) satisfies the following dynamic programming equation:
\begin{align*}\label{V:dpp}
	&V(t,u)=\sup_{\alpha\in\cP(A)^W}\bigg\{\int_{A} f_{t}^0(\Psi_t(u), a, m_t(u)) \alpha(u)(da) \\&\qquad\qquad\qquad\qquad+\sum_{u'\in W}V(t+1,u')\int_{A}P^0_{t+1}(u,a,m_{t}(u),u'){\alpha}(u)(da)\bigg\},\\
	&V(T,u) = g^0(\Psi_T(u),\mu_T(u)).
\end{align*}
Recall that  for every \( t \in \mathbb{T} \setminus \{T\} \), $V^{\lambda_n}$ satisfies 
\begin{align*}
	V^{\lambda_n}(t,u)&=\int_{A} f_{t}^0(\Psi_t(u), a, m^n_t(u))\alpha_t^{\lambda_n}(u)(da) + \lambda_n \mathcal{H}(\alpha_t^{\lambda_n}(u)) \\&\qquad+\sum_{u'\in W}V^{\lambda_n}(t+1,u')\int_{A}P^0_{t+1}(u,a,m^n_{t}(u),u')\alpha^{\lambda_n}_t(u)(da).
\end{align*}
{Taking the limit as $n \to \infty$ and using Assumptions \ref{a1new:P^0} and \ref{a2new:f0g0}, Lemma \ref{lemma:H} and the reasoning analogous to Remark \ref{remark:pi0}, we obtain the following result:
	\begin{equation*}
		V(t,u) = \int_{A} f_{t}^0(\Psi_t(u), a, m_t(u)) \alpha_t(u)(da) + \sum_{u' \in W} V(t+1, u') \int_{A} P^0_{t+1}(u, a, m_t(u), u') \alpha_t(u)(da).
	\end{equation*}}
 Furthermore, $V(T,u) = g^0(\Psi_T(u),\mu_T(u))$  follows directly from taking the limit.

Finally, by a standard verification argument, we can conclude that \( \alpha \) solves the major player's problem  corresponding to \( (\mu, m) \), that is, $\alpha\in \mathbb{A}(\mu,m)$.

\textit{Case 2: $A = \{0, 1\}$.}
It is clear that for any $t\in \mathbb{T}\setminus\{T\}$,
\begin{equation}\label{entropyalpha2}
|\mathcal{H}(\alpha^{\lambda}_t(u))| \leq 2 \max_{p \in [0, 1]} |p \log(p)| = \frac{2}{e}.
\end{equation}
In a similar fashion, we can obtain the uniform boundedness of the value function $V^{\lambda}$ with $\lambda\in[0,1]$. Note that \( \alpha^{\lambda}_t(u)\in \mathcal{P}(\{0, 1\}) \) for any $(t,u)\in \mathbb{T}\setminus\{T\}\times W$. Therefore, there exists a sequence \( \{\lambda_n\}_{n \in \mathbb{N}} \) in \( (0, 1] \) with \( \lambda_n \to 0 \) such that \( (\mu^n, m^n, \tilde{\mu}^n,\tilde{m}^n, \alpha^{\lambda_n}, V^{\lambda_n}) \) converges to \( (\mu, m, \tilde{\mu},\tilde{m}, \alpha, V) \).

Define for any $(t,u)\in\mathbb{T}\setminus\{T\}\times W$ and $i=0,1$,
\begin{align*}
	g(t,u,i):=f^0_{t}(\Psi_t(u), i, m_t(u))+\sum_{u'\in W}V(t+1,u')P^0_{t+1}(u,i,m_{t}(u),u').
\end{align*}
We claim that, for any  $(t,u)\in \mathbb{T}\setminus\{T\}\times W$,
\begin{equation}\label{alphadpp}
	\begin{split}
			&\alpha_t(u)(1)=0,\,\alpha_t(u)(0)=1,\,\ \text{if}\, g(t,u,0)> g(t,u,1),\\
		&\alpha_t(u)(1)=1,\,\alpha_t(u)(0)=0,\,\ \text{if}\, g(t,u,0)< g(t,u,1).
	\end{split}
\end{equation} 
We prove the case when \( g(t, u, 0) > g(t, u, 1) \). Let \( \epsilon := g(t, u, 0) - g(t, u, 1) > 0 \). 
Recall that
\begin{align*}
	&\alpha^{\lambda_n}_t(u)(1)=\frac{e^{\frac{1}{\lambda_n}g^n(t,u,1)}}{\sum_{i=0}^1e^{\frac{1}{\lambda_n}g^n(t,u,i)}}=\frac{1}{1+e^{\frac{1}{\lambda_n}(g^n(t,u,0)-g^n(t,u,1))}},\quad t\in \mathbb{T}\setminus\{T\},
\end{align*}
where 
\begin{align*}
g^n(t,u,i):=f^0_{t}(\Psi_t(u), i, m^n_t(u))+\sum_{u'\in W}V^{\lambda_n}(t+1,u')P^0_{t+1}(u,i,m^n_{t}(u),u').
\end{align*}
Thanks to Assumptions \ref{a1new:P^0} and \ref{a2new:f0g0}, we obtain that $g^n(t,u,i)\rightarrow g(t,u,i)$ as $n\rightarrow\infty$. Then, $g^n(t,u,0)-g^n(t,u,1)\geq\frac{\epsilon}{2}$ when $n$  is large enough. Hence,
\begin{equation*}
\alpha_t(u)(1)=\lim_{n \rightarrow \infty}\alpha^{\lambda_n}_t(u)(1)\leq \lim_{n \rightarrow \infty} \frac{1}{1+e^{\frac{\epsilon}{2\lambda_n}}}=0.
\end{equation*}
We next proceed to show that
\begin{align*}
	&V(t,u) = \sum_{i=0}^{1}f_{t}^0(\Psi_t(u), i, m_t(u))\alpha_t(u)(i) + \sum_{u' \in W, i\in\{0,1\}} V(t+1, u') P^0_{t+1}(u, i, m_t(u), u') \alpha_t(u)(i),\\
	&V(T,u) = g^0(\Psi_T(u),\mu_T(u)),
\end{align*}
which can be verified by taking the limit in the equations below, noting the uniform bound in \eqref{entropyalpha2}:
\begin{align*}
	&V^{\lambda_n}(t,u) = \sum_{i=0}^{1}f_{t}^0(\Psi_t(u), i, m^n_t(u))\alpha^{\lambda_n}_t(u)(i) +\lambda_n\mathcal{H}(\alpha^{\lambda_n}_t(u))\\&\qquad\qquad+ \sum_{u' \in W, i\in\{0,1\}} V^{\lambda_n}(t+1, u') P^0_{t+1}(u, i, m^n_t(u), u') \alpha^{\lambda_n}_t(u)(i),\\
	&V^{\lambda_n}(T,u) = g^0(\Psi_T(u),\mu^n_T(u)).
\end{align*}
This, together with \eqref{alphadpp}, implies that \( V \) satisfies the dynamic programming equation. Therefore, a standard verification argument yields that \( \alpha \) solves the major player's problem associated with \( (\mu, m) \), i.e., \( \alpha \in \mathbb{A}(\mu,m) \).

{\textbf{Step 2: Minor player's problem}

We have shown in step 1 that \( \alpha \in \mathbb{A}(\mu, m) \). Next, we will show that  \( (\tilde{\mu}, \tilde{m}) \in \hat{\Theta}(\mu, m, \alpha) \).

First, recall that for each $n$, $(\tilde{\mu}^n, \tilde{m}^n) \in \hat{\Theta}(\mu^n, m^n, \alpha^n)$, and we have the convergence $(\mu^n, m^n, \alpha^n) \to (\mu, m, \alpha)$ and $(\tilde{\mu}^n, \tilde{m}^n) \to (\tilde{\mu}, \tilde{m})$.

By Lemmas \ref{lemma:Gamma} and \ref{sethatr_con} and the Berge's Maximum Theorem, we obtain that the set-valued mapping $\hat{\Theta}$ is upper hemicontinuous and compact-valued. As its range space $\mathcal{V}$ is a compact Hausdorff space, the Closed Graph Theorem can be applied to deduce that $\hat{\Theta}$ has a closed graph.
 Therefore, we conclude that
\( (\tilde{\mu}, \tilde{m}) \in \hat{\Theta}(\mu, m, \alpha). \)

\textbf{Step 3: Consistency condition}

Finally, we verify the mean-field consistency condition \((\mu, m) \in \Gamma((\tilde{\mu}, \tilde{m}), \alpha, m)\). This is a direct consequence of the closed graph property of the conditioning mapping $\Gamma$, established in Lemma \ref{lemma:conditioning}.

Having verified all three conditions in Definition \ref{def:relaxed}, we conclude that $(\mu,m,\alpha,\tilde{\mu},\tilde{m})$ constitutes a relaxed equilibrium in the original MFG problem.}
\end{proof}

{
\begin{remark}
If  \(A\) is convex and, for each 
\((t,x,m,u,y)\in\mathbb{T}\setminus\{T\}\times S^0\times\cP^{\mathrm{sub}}(S)\times W\times\mathbb{R}^{|W|}\), the map
\[
    a \mapsto f^0_t(x,a,m)+\sum_{u' \in W} P^0_{t+1}(u,a,m,u')\,y_{u'}
\]
is strictly concave, then \(E_t(y,u,m)\) defined in \eqref{E} is a singleton. Consequently, the optimal relaxed control for the major player reduces to a strict control (or pure strategy).
\end{remark}}
{\section{MFG of Controls for the Minor Players}\label{sect:6}
In this section, we extend the entropy regularization method to the setting where the minor players face optimal control problems. Specifically, we consider a system consisting of a major player, whose path process is denoted by $U = (U_k)_{k \in \mathbb{T}}$, and a representative minor player with state process $X = (X_k)_{k \in \mathbb{T}}$. For  $k\in\mathbb{T}\setminus\{T\}$, the state transitions are governed by the following controlled stochastic kernels:
{\small\begin{equation}\label{transition_control}
  	\begin{split}
  			&U_{k+1}\sim \pi^0_{k+1}(U_{k},du';\alpha^0_k(U_{k}),\mu_k(U_{k})):= \int_{A^0}P^0_{k+1}(U_{k},a,\mu_{k}(U_{k}),du')\,\alpha^0_{k}(U_{k})(da),\\
  		&X_{k+1}\sim \pi^{1}_{k+1}(X_{k},U_{k},dx';\alpha^1_k(X_k,U_k),\mu_{k}(U_{k})):=\int_{A^1}P^1_{k+1}(X_k,U_{k},a,\mu_{k}(U_{k}),dx')\,\alpha^1_{k}(X_k,U_{k})(da).
  	\end{split}
\end{equation}}
Here, \(A^0\) and \(A^1\) denote the compact action sets of the major and minor players, respectively.  
The policies of the major and minor players are specified by measurable maps
$\alpha^0_k: W \to \mathcal{P}(A^0)$ and $\alpha^1_k: S \times W \to \mathcal{P}(A^1)$, respectively.
The  sets of admissible policies for the major and minor players are denoted by \(\mathcal{A}^0\) and \(\mathcal{A}^1\), respectively.

The mean-field interaction is captured by the collection $\mu = (\mu_k)_{k=0}^T$. For each $k\in\mathbb{T}\setminus\{T\}$, the term $\mu_k: W \to \mathcal{P}(S \times A^1)$ describes the minor player's conditional state-action distribution at time $k$, given the major player's information $U_k$. The terminal term $\mu_T: W \to \mathcal{P}(S)$ represents the conditional distribution of the minor player's terminal state $X_T$, given $U_T$.

The major player's objective is to maximize the reward functional:
\begin{equation*}
J^0(\alpha^0;\mu) = \mathbb{E} \left[ \sum_{k=0}^{T-1} \int_{A^0} f_{k}^0(U_{k},a,\mu_k(U_k))\,\alpha^0_{k}(U_k)(da) + g^0(U_T, \mu_{T}(U_T)) \right],
\end{equation*}
and the representative minor player aims to solve a stochastic control problem with the reward functional:
\begin{equation*}
J^1(\alpha^1;\alpha^0,\mu) = \mathbb{E} \left[ \sum_{k=0}^{T-1} \int_{A^1} f_{k}^1(X_k,U_{k},a,\mu_k(U_k))\,\alpha^1_{k}(X_k,U_k)(da) + g^1(X_T,U_T,\mu_{T}(U_T)) \right].
\end{equation*}
\subsection{Definitions of equilibrium}
The major player's best-response mapping \(\mathbb{A}(\mu): \mathcal{C}_{\mu}\rightarrow 2^{\mathcal{A}^0}\) is defined as :
\begin{align*}
    \mathbb{A}(\mu) := \operatorname*{arg\,max}_{\alpha^0 \in \mathcal{A}^0} J^0(\alpha^0;\mu),
\end{align*}
where \(\mathcal{C}_{\mu} := \left(\prod_{k \in \mathbb{T}\setminus\{T
\}} \mathcal{P}(S\times A^1)^{W}\right)\times\cP(S)^{W}\) is the space of mean-field flows.   

For any given \(\mu \in \mathcal{C}_{\mu}\) and \(\alpha^0 \in \mathcal{A}^0\), the minor player's  value function at time $t \in \mathbb{T}$ is defined as:
{\begin{align*}
    V^{1,\alpha^0,\mu}_t(x,u)=\sup_{\alpha^1 \in \mathcal{A}^1_t} \mathbb{E} \bigg[ \sum_{s=t}^{T-1} &\int_{A^1} f_{s}^1(X_s,U_{s},a,\mu_s(U_s))\,\alpha^1_{s}(X_s,U_s)(da)\\ &\quad\quad+ g^1(X_T,U_T,\mu_{T}(U_T)) \,\bigg|\, (X_t,U_t)=(x,u) \bigg].
\end{align*}}Here, \(\mathcal{A}^1_t\) denotes the set of the representative minor player's admissible controls from time \(t\) onwards.
Under mild continuity assumptions (see Assumption \ref{newa}), the value function $V^{1,\alpha^0,\mu}$ is well-defined and serves as the unique solution to the dynamic programming equation. The associated Bellman optimality operator $T_t^{\alpha^0, \mu}: C_b(S \times W) \to C_b(S \times W)$ is given by:
\begin{equation*}
    (T_t^{\alpha^0,\mu} v)(x,u) = \max _{a \in A^1}\left\{ f^1_t(x,u,a,\mu_t(u))+\int_{S\times W} v(x',u') \,\pi_{t+1}(x, u,a;\alpha^0_t(u),\mu_t(u);dx',du')\right\}.
\end{equation*}
where
\begin{align*}
    \pi_{t+1}(x,u,a;\alpha^0_t(u),\mu_t(u);dx',du') := \pi_{t+1}^0(u,du';\alpha^0_{t}(u),\mu_{t}(u))\,P^1_{t+1}(x,u,a,\mu_{t}(u),dx'),\quad t\in\mathbb{T}\setminus\{T\}.
\end{align*}
The value function satisfies the dynamic programming principle, expressed via the Bellman equation for all $t \in \mathbb{T}\setminus\{T\}$:
\[
V^{1,\alpha^0,\mu}_t(x,u) = (T_t^{\alpha^0,\mu}V^{1,\alpha^0,\mu}_{t+1})(x,u),
\]
with the terminal condition $V^{1,\alpha^0,\mu}_T(x,u) = g^1(x,u,\mu_T(u))$.

To analyze the equilibrium, we reformulate the minor player's problem in the space of state-action measure flows, $\mathcal{V} \coloneqq \left(\prod_{k\in\mathbb{T}\setminus\{T\}} \mathcal{P}(S \times W\times A^1)\right)\times\mathcal{P}(S \times W)$. We define two key sets:
\begin{itemize}
    \item The set of dynamically consistent flows is given by:
\begin{align*}
    \mathbb{C}(\alpha^0,\mu)=&\bigg\{v\in\mathcal{V}: v_{0}|_{S\times W}(dx,du)=m_0^*(dx,du),\\& v_{t+1}|_{S\times W}(dx,du)=\int_{S \times W\times A^1} \pi_{t+1}(x',u',a;\alpha^0_t(u'),\mu_t(u');dx,du) v_t(dx', du', da)\bigg\}
\end{align*}
    \item The set of optimal flows is defined by the dynamic programming principle
{\small\begin{align*}
\mathbb{B}(\alpha^0,\mu) =\bigg\{&v \in\mathcal{V}: \forall t\in\mathbb{T}\setminus\{T\}, v_t\Big(\Big\{(x,u,a): f_t^1(x,u,a,\mu_t(u)) 
\\&+ \int_{S\times W} V^{1,\alpha^0,\mu}_{t+1}(x',u')\, \pi_{t+1}(x,u,a;\alpha^0_t(u),\mu_t(u);dx',du')
= \left(T_t^{\alpha^0,\mu}V^{1,\alpha^0,\mu}_{t+1}\right)(x,u)\Big\}\Big)=1\bigg\} .
\end{align*}}
\end{itemize}
Note that the set \(\mathbb{C}(\alpha^0,\mu)\) characterizes the \emph{feasible state-action flows} consistent with the system dynamics, whereas \(\mathbb{B}(\alpha^0,\mu)\) characterizes the \emph{optimality condition} for the policy obtained by disintegrating the state-action flow in response to \((\alpha^0,\mu)\). 
Accordingly, the set of \emph{optimal state-action flows} for the minor player is given by
\begin{align*}
    \mathbb{D}(\alpha^0,\mu) := \mathbb{B}(\alpha^0,\mu) \cap \mathbb{C}(\alpha^0,\mu),
\end{align*}
which collects all optimal state-action measure flows, given the major player's control \(\alpha^0\) and the mean field process \(\mu\).

To relate these flows back to the mean-field term, we introduce the conditioning mapping
\(\Gamma: \mathcal{V} \to 2^{\mathcal{C}_\mu}\) that  
for any flow \(v = (v_t)_{t \in \mathbb{T}} \in \mathcal{V}\), its image \(\Gamma(v)\) is defined by
\begin{align*}
    \Gamma(v) := \bigg\{ \mu  \in \mathcal{C}_\mu: \;\; 
    &v_t(dx,du,da) = \mu_t(u)(dx,da)\,(v_t|_{W})(du), \;\; \forall t \in \mathbb{T}\setminus\{T\} \\
    &v_T(dx,du) = \mu_T(u)(dx)\,(v_T|_{W})(du)\bigg\}.
\end{align*}
By the disintegration theorem, the stochastic kernel \(\mu_t(u)\) is uniquely determined for 
\(u \in \operatorname{supp}(v_t|_{W})\).  
However, for any state \(u \in W\) such that \((v_t|_{W})(\{u\}) = 0\), the conditional distribution 
\(\mu_t(u)\) is not constrained by the disintegration and may be chosen arbitrarily from \(\mathcal{P}(S\times A^1)\) (for $t<T$) or $\mathcal{P}(S)$ (for $t=T$).
\begin{definition}[Relaxed Equilibrium]
A tuple \((\mu^*,\alpha^{0,*},v^*)\) is a relaxed equilibrium if it satisfies the following three conditions simultaneously:
\begin{enumerate}
    \item  Given the mean-field terms \(\mu^*\), the major player's relaxed control \(\alpha^{0,*} \in \mathcal{A}^0\) attains the optimality:
    \[ \alpha^{0,*} \in \mathbb{A}(\mu^*).\]

    \item  Given \((\mu^*,  \alpha^{0,*})\), the minor player's state-action flow \(v^* \in \mathcal{V}\) attains the optimality: 
    \[ v^* \in \mathbb{D}(\alpha^{0,*},\mu^*).\]

    \item The mean-field terms \(\mu^*\) is consistent with the minor player's state-action flow \(v^*\):
    \begin{align*}
        &v^*_t(dx, du,da) = \mu^*_t(u)(dx,da)  \, (v^*_t|_{W})(du), \quad\forall t \in \mathbb{T}\setminus\{T\}, \\
        &v^*_T(dx, du) = \mu^*_T(u)(dx) \, (v^*_T|_{W})(du).
    \end{align*}
\end{enumerate}
\end{definition}
With the preceding definitions, we can characterize the relaxed equilibrium in this major-minor MFG in terms of a fixed-point problem. Let us define the set-valued mapping $\Phi: \mathcal{C}_\mu \rightarrow 2^{\mathcal{C}_\mu}$ as follows:
\begin{align*}
    \Phi(\mu) \coloneqq \Gamma\left(\bigcup_{\alpha^0 \in \mathbb{A}(\mu)} \mathbb{D}(\alpha^0,\mu)\right).
\end{align*}
The set of relaxed equilibria for the major-minor MFG is precisely the set of fixed points of this mapping. 
Similarly, the fixed-point mapping \(\Phi\) may fail to be convex-valued. 
However, the required convexity property is recovered when the major player's control is fixed, as formalized in the following lemma.
\begin{lemma} \label{lemma:control}
\begin{itemize}
    \item [(i)] For any fixed tuple $(\alpha^0,\mu)\in\cA^0\times\mathcal{C}_\mu$  such that $\mathbb{D}(\alpha^0,\mu)$ is nonempty, the image set $\Gamma(\mathbb{D}(\alpha^0,\mu))$ is a nonempty convex subset of $\mathcal{C}_\mu$.
   \item[(ii)] The conditioning mapping $\Gamma: \mathcal{V} \rightarrow 2^{\mathcal{C}_\mu}$ has a closed graph.
\end{itemize}
\end{lemma}
\begin{proof}
The proofs are identical to those of Lemmas \ref{lem:properties_gamma_theta} and \ref{lemma:conditioning}, and are hence omitted.
\end{proof}

 Similar to the stopping case, we are motivated to consider whether the entropy regularization method can be applied in this control context. To this end, we introduce the major player's regularized reward functional for a regularization parameter $\lambda > 0$:
\begin{align*}
    J^0_{\lambda}(\alpha^0;\mu)
    \coloneqq J^0(\alpha^0;\mu)
    + \lambda\,\mathbb{E}\!\left[ \sum_{t=0}^{T-1} \mathcal{H}(\alpha_t^0(U_t)) \right].
\end{align*}
The major player's regularized best-response control is then defined as
\begin{align*}
    \alpha^0_{\lambda}(\mu) 
    \coloneqq \operatorname*{arg\,max}_{\alpha^0 \in \mathcal{D}} J^0_{\lambda}(\alpha^0;\mu),
\end{align*}
where $\mathcal{D}$ denotes the set of all regularized controls as in Definition~\ref{def:regular_con}.  A $\lambda$-regularized equilibrium is then characterized as a fixed point of the regularized equilibrium mapping 
$\Phi_\lambda: \mathcal{C}_\mu \rightarrow 2^{\mathcal{C}_\mu}$
 given by
\[
    \Phi_\lambda(\mu) \coloneqq \Gamma\big(\mathbb{D}(\alpha^0_{\lambda}(\mu),\mu)\big).
\]
\subsection{Technical assumptions}\label{subsect:assump}
We introduce the following standing assumptions within the control context. As the assumptions on the major player's problem are analogous to those in Section~\ref{sect:technical assumptions}, 
we only focus on the conditions related to the minor player's problem. 
\begin{assumption}\label{newa}
	\begin{enumerate}[\upshape (i)]
    \item \emph{(Minor Player Transition Kernel):} For each fixed $(k,u) \in (\mathbb{T} \setminus \{T\}) \times W$, the mapping
    \(
        (x, a, m) \mapsto P^1_{k+1}(x, u, a, m, dx)
    \)
    from $S \times A^1 \times \mathcal{P}(S\times A^1)$ to $\mathcal{P}(S)$ is continuous.

    \item \emph{(Minor Player Running Reward):} For each fixed $(t,u) \in (\mathbb{T} \setminus \{T\}) \times W$, the running reward function
    \(
        (x, a, m) \mapsto f^1_t(x, u, a, m)
    \)
    is continuous on $S \times A^1 \times \mathcal{P}(S\times A^1)$.

    \item \emph{(Minor Player Terminal Reward):} For each fixed $u \in W$, the terminal reward function
    \(
        (x, m) \mapsto g^1(x, u, m)
    \)
    is continuous on $S \times \mathcal{P}(S)$.
\end{enumerate}
\end{assumption}
\begin{remark}\label{remark:6.3}
Under Assumption~\ref{newa}, for any fixed pair $(\alpha^0, \mu)\in \mathcal{A}^0\times\mathcal{C}_{\mu}$, it is clear that the minor player's value function $V^{1,\alpha^0,\mu}$ is well-defined and is the unique continuous solution to the dynamic programming equation. Consequently, by a standard measurable selection argument, the set of optimal state-action flows $\mathbb{D}(\alpha^0, \mu)$ is nonempty.
\end{remark}
\subsection{Existence of relaxed equilibrium}
We again investigate the existence of a relaxed equilibrium by first establishing the existence of a regularized equilibrium and then taking the limit as the regularization vanishes.

To this end, we apply the Kakutani–Fan–Glicksberg fixed point theorem to the set-valued mapping  $\Phi_\lambda: \mathcal{C}_\mu \rightarrow 2^{\mathcal{C}_\mu}$ given by \[ \Phi_\lambda(\mu) \coloneqq \Gamma\big(\mathbb{D}(\alpha^0_{\lambda}(\mu),\mu)\big). \]
By Lemma~\ref{lemma:control} and Remark~\ref{remark:6.3}, $\Phi_\lambda$ has nonempty convex values; it thus remains to verify that $\Phi_\lambda$ has a closed graph. 
\begin{lemma}\label{lemma:control_closed}
    Under  Assumption \ref{a1new:P^0} (i), Assumption \ref{a2new:f0g0} and  Assumption~\ref{newa}, the set-valued mapping $\mathbb{D}: \mathcal{A}^0 \times \mathcal{C}_\mu \to 2^{\mathcal{V}}$ given by 
    \begin{align*}
    \mathbb{D}(\alpha^0,\mu) := \mathbb{B}(\alpha^0,\mu) \cap \mathbb{C}(\alpha^0,\mu),
\end{align*}
has a closed graph.
\end{lemma}
\begin{proof}
    Let $(\alpha^{0,n}, \mu^n) \to (\alpha^0, \mu)$ and $v^n \in \mathbb{D}(\alpha^{0,n}, \mu^n)$ with $v^n \to v$.  To prove $\mathbb{D}$ has a closed graph, it is sufficient to prove $v \in \mathbb{D}(\alpha^0, \mu)$.

We first show $v \in \mathbb{C}(\alpha^0, \mu)$. For each $n$ and $t \in \mathbb{T}\setminus\{T\}$, we have
\[
v^n_{t+1}|_{S\times W}(dx,du)
= \int_{S \times W\times A^1} \pi_{t+1}(x',u',a;\alpha^{0,n}_t(u'),\mu^n_t(u');dx,du)\, v^n_t(dx', du', da).
\]
To verify that the same equation holds in the limit, we test against an arbitrary $f \in C_b(S\times W)$. Define
\[
G_{t+1}(x',u',a,\alpha,\mu) \coloneqq \int_{S\times  W} f(x,u)\,
\pi_{t+1}(x',u',a;\alpha,\mu;dx,du).
\]
By the continuity assumptions on the kernels and the argument in Remark~\ref{remark:pi0}, $G_{t+1}$ is bounded and continuous in all variables. As $S \times W \times A^1 \times \cP(A^0) \times \cP(S\times A^1)$ is compact, $G_{t+1}$ is in fact uniformly continuous. Using this notation, the consistency condition for $v^n$ becomes:
\[
\int_{S\times W} f(x,u) \, v^n_{t+1}|_{S\times W}(dx,du) = \int_{S \times W\times A^1} G_{t+1}(x', u', a, \alpha^{0,n}_t(u'), \mu^n_t(u')) \, v^n_t(dx', du', da).
\]
By taking the limit as $n \to \infty$, the left-hand side converges to $\int f \, dv_{t+1}|_{S\times W}$ by the definition of weak convergence of $v^n_{t+1}$. For the right-hand side, as $W$ is finite and $G_{t+1}$ is uniformly continuous, we obtain
\begin{align}\label{equ:convergence}
  &\lim _{n \rightarrow \infty}\int_{S \times W\times A^1} G_{t+1}(x',u',a,\alpha^{0,n}_t(u'),\mu^n_t(u')) v^n_t(dx', du', da)\\
    &=\int_{S \times W\times A^1} G_{t+1}(x',u',a,\alpha^{0}_t(u'),\mu_t(u')) v_t(dx', du', da)\nonumber\\
    &=  \int_{S \times W\times A^1} \int_{S\times  W} f(x,u)  \pi_{t+1}(x',u',a;\alpha^{0}_t(u'),\mu_t(u');dx,du) v_t(dx', du', da).\nonumber
\end{align}
 Therefore, 
$$
v_{t+1}|_{S\times W}(dx,du)=\int_{S \times W\times A^1} \pi_{t+1}(x',u',a;\alpha^{0}_t(u'),\mu_t(u');dx,du) v_t(dx', du', da).
$$
from which we deduce that $v \in \mathbb{C}(\alpha^0, \mu)$.

It remains to prove that $v \in \mathbb{B}(\alpha^0, \mu)$.  
For notational convenience, set $V_t^n := V_t^{1,\alpha^{0,n},\mu^n}$ and $V_t := V_t^{1,\alpha^{0},\mu}$.  
By Proposition~3.9 in \cite{saldi_markov--nash_2018}, together with the discussion following its proof, it suffices to show that $V_t^n \to V_t$ uniformly for each $t \in \mathbb{T}$.

Recall that both $V^n$ and $V$ satisfy the dynamic programming equation:
\begin{align*}
    &V^n_t(x,u) = (T^{\alpha^{0,n},\mu^n}_t V^n_{t+1})(x,u), \quad V^n_T(x,u) = g^1(x,u,\mu^n_T(u)),\\
    &V_t(x,u) = (T^{\alpha^{0},\mu}_t V_{t+1})(x,u), \quad V_T(x,u) = g^1(x,u,\mu_T(u)).
\end{align*}
 We prove the uniform convergence of the value functions, $V^n_t \to V_t$, by backward induction. At the terminal time $T$, since $g^1$ is uniformly continuous and $W$ is finite, we have
\[
    \lim_{n\to\infty} \sup_{(x,u) \in S \times W} |V^n_T(x,u) - V_T(x,u)| = 0.
\]
Now, assume the claim holds for time $t+1$, i.e., $\sup_{(x',u') \in S \times W} |V^n_{t+1}(x',u') - V_{t+1}(x',u')| \to 0$ as $n \to \infty$. Then, for time $t$, we have
\begin{align*}
    \sup_{(x,u)} &|V^n_t(x,u) - V_t(x,u)| 
    \\&= \sup_{(x,u)} \bigg| \max_{a \in A^1} \Big[ f^1_t(x,u,a,\mu_t^n(u)) + \langle V^n_{t+1}, \pi_{t+1}(x,u,a;\alpha^{0,n}_t(u),\mu^n_t(u)) \rangle \Big] \\
    & - \max_{a \in A^1} \Big[ f^1_t(x,u,a,\mu_t(u)) + \langle V_{t+1}, \pi_{t+1}(x,u,a;\alpha^{0}_t(u),\mu_t(u)) \rangle \Big] \bigg|  \\
    &\le \sup_{(x,u,a)} \left| f^1_t(x,u,a,\mu_t^n(u)) - f^1_t(x,u,a,\mu_t(u)) \right| \\
    &+\sup_{(x,u,a)} \left|  \langle V^n_{t+1}, \pi_{t+1}(x,u,a;\alpha^{0,n}_t(u),\mu^n_t(u)) \rangle - \langle V_{t+1}, \pi_{t+1}(x,u,a;\alpha^{0}_t(u),\mu_t(u)) \rangle \right| \\
    &\le \sup_{(x,u,a)} \left| f^1_t(x,u,a,\mu_t^n(u)) - f^1_t(x,u,a,\mu_t(u)) \right| + \sup_{(x',u')} |V^n_{t+1}(x',u') - V_{t+1}(x',u')| \\
    & + \sup_{(x,u,a)} \bigg| \int_{S \times W} V_{t+1}(x',u')  \pi_{t+1}(x,u,a;\alpha^{0,n}_t(u),\mu^n_t(u);dx',du')\\ &-\int_{S \times W} V_{t+1}(x',u') \pi_{t+1}(x,u,a;\alpha^{0}_t(u),\mu_t(u);dx',du')  \bigg|,
\end{align*}
where $\langle V, \pi \rangle$ denotes the integral of $V$ with respect to $\pi$. Then, by the uniform continuity of \(f^1_t\) and \(\pi_{t+1}\), together with the finiteness of \(W\) and the induction hypothesis, we obtain
\[
    \lim_{n\to\infty} \sup_{(x,u) \in S \times W} |V^n_t(x,u) - V_t(x,u)| = 0.
\]
 By backward induction, the uniform convergence holds for all $t \in \mathbb{T}$, which implies $v \in \mathbb{B}(\alpha^0, \mu)$. Hence, $v \in \mathbb{D}(\alpha^0, \mu)$ as desired. 
\end{proof}

With the above lemma, the closed graph property of $\Phi_\lambda$ follows from a composition of continuous and closed-graph mappings. Let $(\mu^n, \mu'^n) \to (\mu, \mu')$ be a sequence in the graph of $\Phi_\lambda$. There exists $v^n \in \mathbb{D}(\alpha^{0}_\lambda(\mu^n), \mu^n)$ such that $\mu'^n \in \Gamma(v^n)$. By compactness of $\mathcal{V}$, we can extract a convergent subsequence $v^n \to v$. By the closed graph of $\mathbb{D}$ (Lemma \ref{lemma:control_closed}) and the continuity of $\alpha^0_\lambda$ (Remark \ref{remark:alphacontinu}), we have $v \in \mathbb{D}(\alpha^0_\lambda(\mu), \mu)$. By the closed graph of $\Gamma$ (Lemma~\ref{lemma:control}), we have $\mu' \in \Gamma(v)$. This chain of arguments implies $\mu' \in \Phi_\lambda(\mu)$, so $\Phi_\lambda$ has a closed graph. Consequently, a $\lambda$-regularized equilibrium exists for every $\lambda > 0$.

Let \( \{\lambda_n\}_{n \geq 1} \) be a sequence in \((0, 1]\) with \(\lambda_n \downarrow 0\). For each \(n\), let \( (\mu^n, \alpha^{0,n}, v^n) \) be a $\lambda_n$-regularized equilibrium tuple, where $\alpha^{0,n} = \alpha^0_{\lambda_n}(\mu^n)$ and $v^n \in \mathbb{D}(\alpha^{0,n}, \mu^n)$ such that $\mu^n \in \Gamma(v^n)$. In view that $\mathcal{C}_{\mu} \times \mathcal{A} \times \mathcal{V}$ is compact, the sequence of equilibrium tuples admits a convergent subsequence converging to a limit point \( (\mu, \alpha, v) \). 

Analogous to Theorem~\ref{thm:existencerelaxed},  one can show that $\alpha \in \mathbb{A}(\mu)$. By the closed graph properties of $\mathbb{D}$ and $\Gamma$, it then follows that $v \in \mathbb{D}(\alpha, \mu)$ and $\mu \in \Gamma(v)$, thereby establishing the existence of a relaxed equilibrium in the original major-minor MFG.

}

\appendix
\section{Proofs of Lemmas}\label{appendix:proof}
\begin{proof}[Proof of Lemma \ref{lem:properties_gamma_theta}]
     Let $K \coloneqq \Gamma(\hat{\Theta}(\mu, m, \alpha), \alpha, m)$.  For a fixed $\alpha$ and $m$, the marginal law of the major player's state, $p_t(\cdot; \alpha, m)$, is uniquely determined and fixed.
Let $(\mu^1, m^1)$ and $(\mu^2, m^2)$ be two points in $K$. By definition, there exist occupation measure pairs $(\tilde{\mu}^1, \tilde{m}^1), (\tilde{\mu}^2, \tilde{m}^2) \in \hat{\Theta}(\mu, m, \alpha)$ such that $(\mu^1, m^1) \in \Gamma((\tilde{\mu}^1, \tilde{m}^1), \alpha, m)$ and $(\mu^2, m^2) \in \Gamma((\tilde{\mu}^2, \tilde{m}^2), \alpha, m)$.
Let $\lambda \in [0,1]$. Consider the convex combinations:
\begin{align*}
    (\tilde{\mu}^\lambda, \tilde{m}^\lambda) &\coloneqq \lambda(\tilde{\mu}^1, \tilde{m}^1) + (1-\lambda)(\tilde{\mu}^2, \tilde{m}^2) \\
    (\mu^\lambda, m^\lambda) &\coloneqq \lambda(\mu^1, m^1) + (1-\lambda)(\mu^2, m^2)
\end{align*}
Since $\hat{\Theta}(\mu, m, \alpha)$ is a convex set,  $(\tilde{\mu}^\lambda, \tilde{m}^\lambda)$ is also in $\hat{\Theta}(\mu, m, \alpha)$. We now show that \((\mu^\lambda, m^\lambda)\) is a disintegration of \((\tilde{\mu}^\lambda, \tilde{m}^\lambda)\) with respect to the fixed marginal law \(p_t(\cdot; \alpha, m)\).
 For any $t$,
\begin{align*}
    \tilde{\mu}^\lambda_t(dx, du) &= \lambda \tilde{\mu}^1_t(dx, du) + (1-\lambda) \tilde{\mu}^2_t(dx, du) \\
    &= \lambda \mu^1_t(u)(dx) p_t(du; \alpha, m) + (1-\lambda) \mu^2_t(u)(dx) p_t(du; \alpha, m) \\
    &= (\lambda \mu^1_t(u) + (1-\lambda) \mu^2_t(u))(dx) p_t(du; \alpha, m) \\
    &= \mu^\lambda_t(u)(dx) p_t(du; \alpha, m).
\end{align*}
A similar calculation holds for $\tilde{m}^\lambda_t$. This shows that $(\mu^\lambda, m^\lambda) \in \Gamma((\tilde{\mu}^\lambda, \tilde{m}^\lambda), \alpha, m)$. Since $(\tilde{\mu}^\lambda, \tilde{m}^\lambda) \in \hat{\Theta}(\mu, m, \alpha)$, it follows that $(\mu^\lambda, m^\lambda) \in K$. Therefore, $K$ is convex.
\end{proof}

\begin{proof}[Proof of Lemma \ref{lemma:r[m,a]}]
	 Relative compactness follows from the fact that $\mathcal{P}^{\text {sub }}(S\times W)$ is compact and henceforth $\mathcal{V}$ is also compact. Let us check that $\mathcal{R}[m;\alpha]$ is closed. Consider a sequence $\left(\tilde{\mu}^n, \tilde{m}^n\right)_{n \geq 1} \subset \mathcal{R}[m;\alpha]$ converging to some $(\tilde{\mu}, \tilde{m})$ and we claim that $(\tilde{\mu}, \tilde{m}) \in \mathcal{R}[m;\alpha]$. For all $\varphi \in C_b(\mathbb{T} \times S \times W)$ and $n \geq 1$,
	\begin{align*}
	\sum_{t=0}^{T}  \int_{S\times W} \varphi(t, x, u) \tilde{\mu}^n_t(dx,du)
	&=  \int_{S\times W} \varphi(0, x, u) m_0^*(dx,du) \\&+ \sum_{t=0}^{T-1}  \int_{S\times W} \mathcal{L}(\varphi)(t, x, u;\alpha_t(u),m_{t}(u))\tilde{m}^n_t(dx,du).
\end{align*}
	In view that for each $t \in \mathbb{T} \backslash\{T\}$ the function $(x,u) \mapsto \mathcal{L}(\varphi)(t, x, u;\alpha_t(u),m_{t}(u))$ is continuous (by Remark \ref{remark:pi}) and bounded, we have
	$$
		\lim _{n \rightarrow \infty}  \int_S \mathcal{L}(\varphi)(t, x, u;\alpha_t(u),m_{t}(u))\tilde{m}^n_t(dx,du)= \int_S \mathcal{L}(\varphi)(t, x, u;\alpha_t(u),m_{t}(u))\tilde{m}_t(dx,du) .
	$$
	Therefore, by passing directly to the limit, we conclude that $(\tilde{\mu}, \tilde{m}) \in \mathcal{R}[m;\alpha]$.
\end{proof}
\begin{proof}[Proof of Lemma \ref{lemma:Gamma}]
 Let $\left(\bar{\mu}^n, \bar{m}^n\right)_{n \geq 1} \subset \mathcal{V}$ and $\left(\mu^n, m^n\right)_{n \geq 1} \subset \mathcal{C}_{\mu}\times\mathcal{C}_{m}$ be two sequences converging to $(\bar{\mu}, \bar{m})\in\mathcal{V}$ and $(\mu, m) \in  \mathcal{C}_{\mu}\times\mathcal{C}_{m}$ respectively. {{By Assumption~\ref{a3:minor}, and using the same reasoning as in Remark~\ref{remark:pi0} and the fact that \(W\) is a finite set, we obtain that for all \(t \in \mathbb{T} \setminus \{T\}\),
	$$
	\int_{S\times W} f_t\left(x, u, m_t^n(u)\right) \bar{m}_t^n(dx,du) \underset{n \rightarrow \infty}{\longrightarrow} \int_{S\times W} f_t\left(x, u, m_t(u)\right) \bar{m}_t(dx,du),
 	$$
	and for all $t \in \mathbb{T}$,
	$$
	\int_{S\times W} g_t\left(x, u, \bar{\mu}_t^n(u)\right) \mu_t^n(d x,du) \underset{n \rightarrow \infty}{\longrightarrow} \int_{S\times W} g_t\left(x, u, \bar{\mu}_t(u)\right) \mu_t(d x,du) .
	$$}}
Therefore, it follows that $((\bar{\mu}, \bar{m}),(\mu, m)) \mapsto J^1(\bar{\mu}, \bar{m};\mu, m)$ is continuous.  
\end{proof}
\begin{proof}[Proof of Lemma \ref{lemma:conditioning}]
Let $(((\tilde{\mu}^n, \tilde{m}^n), \alpha^n, m^n), (\mu'^n, m'^n))$ be a sequence in the graph of $\Gamma$ converging to $(((\tilde{\mu}, \tilde{m}), \alpha, m), (\mu', m'))$. We need to show that the limit point is also in the graph.

First, by Remark~\ref{remark:pi0} and an inductive argument, the pointwise convergence of $(\alpha^n, m^n)$ to $(\alpha, m)$ implies the weak convergence of the major player's marginal law: $p_t(\cdot; \alpha^n, m^n) \to p_t(\cdot; \alpha, m)$ for each $t \in \mathbb{T}$.

The condition $(\mu'^n, m'^n) \in \Gamma((\tilde{\mu}^n, \tilde{m}^n), \alpha^n, m^n)$ means that for any $t$ and for all bounded continuous test functions $f: S\times W \to \mathbb{R}$, the following equality holds:
\begin{equation*}\label{eq:disintegration_test_n}
    \int_{S \times W} f(x,u) \tilde{\mu}^n_t(dx, du) =   \int_{S\times W} f(x,u) \, \mu'^n_t(u)(dx)  p_t(du; \alpha^n, m^n).
\end{equation*}
 Passing to the limit as $n \to \infty$, we obtain
\[  \int_{S \times W} f(x,u) \tilde{\mu}_t(dx, du) =   \int_{S\times W} f(x,u) \, \mu'_t(u)(dx)  p_t(du; \alpha, m). \]
 It follows that
\[
\tilde{\mu}_t(dx,du) = \mu'_t(u)(dx)\, p_t(du; \alpha, m).
\] An entirely analogous argument applies to $(\tilde{m}^n_t, m'^n_t)$. We conclude that $(\mu', m') \in \Gamma((\tilde{\mu}, \tilde{m}), \alpha, m)$.  
Hence the graph of $\Gamma$ is closed.
\end{proof}

\section{Probability Space Setup}\label{appendixB}
In this section, we provide some detailed construction of the filtered probability space that has been used in the proof of Lemma \ref{setR*contin}. 
Let us consider a probability space \((\Theta, \mathcal{B}_{\Theta}, \mathbb{P}_{\vartheta})\), where we assume that \(\mathbb{P}_{\vartheta}\) is atomless. This assumption guarantees the existence of Borel measurable functions \(h: \Theta \to [0,1]\) such that \(h\) is uniformly distributed on \([0,1]\) when viewed as a random variable on \((\Theta, \mathcal{B}_{\Theta}, \mathbb{P}_{\vartheta})\). These uniform random variables, constructed via the functions \(h\), will be employed repeatedly along with the following classical result from measure theory, which can be found in \cite{Carmona2023}.
\begin{lemma}(Blackwell-Dubins lemma)\label{lemma:a1}
	 For any Polish space $B$, there exists a measurable function $\rho_B: \mathcal{P}(B) \times[0,1] \mapsto B$, which we shall call the Blackwell-Dubins function of the space $B$, satisfying:
	 \begin{itemize}
	 	\item[(i)] for each $v \in \mathcal{P}(B)$ and each uniform random variable $U \sim U(0,1)$, the $B$-valued random variable $\rho_B(\nu, U)$ has distribution $\nu$;
	 	
	 	\item[(ii)]  for almost every $u \in[0,1]$, the function $v \mapsto \rho_B(v, u)$ is continuous for the weak topology of $\mathcal{P}(B)$.
	 \end{itemize}
\end{lemma}
 Consider the canonical probability space \((\Omega^c, \mathcal{F}^c, \mathbb{P}^c)\), defined as follows:  
\begin{align*}
	 \Omega^c := \underbrace{\Theta \times \Theta \times \cdots \times \Theta}_{T+1}, \quad\quad
	 \mathcal{F}^c := \underbrace{\mathcal{B}_{\Theta} \times \mathcal{B}_{\Theta} \times \cdots \times \mathcal{B}_{\Theta}}_{T+1},
\end{align*}
and equipped with the product probability measure  
\[
\mathbb{P}^c := \underbrace{\mathbb{P}_{\vartheta} \otimes \mathbb{P}_{\vartheta} \otimes \cdots \otimes \mathbb{P}_{\vartheta}}_{T+1}.
\]  
A generic element \(\omega \in \Omega^c\) is represented as  $\omega = (\theta_0, \theta_1, \ldots, \theta_T)$, and the random variables \(\vartheta_n\) are realized as coordinate mappings, defined for \(n \in \mathbb{T}\) by \(\vartheta_n(\omega) = \theta_n\). 
We  introduce the filtration \(\mathbb{F} = (\mathcal{F}_n)_{n \in \mathbb{T}}\), where  
$\mathcal{F}_n = \sigma\{\theta_0, \ldots, \theta_n\}$, $n \in \mathbb{T}$. Consider the following system:
\begin{align*}
	&(X_0,U_0)=(X^n_0,U^n_0)=\rho_{B}(m^*_{0},h(\vartheta_0)),\quad n\in\mN,\\
	&(X_{t+1},U_{t+1})=\rho_B(\pi_{t+1}(X_t,U_t),h(\vartheta_{t+1})),\quad t\in \mathbb{T}\setminus\{T\},\\
	&(X^n_{t+1},U^n_{t+1})=\rho_B(\pi^n_{t+1}(X^n_t,U^n_t),h(\vartheta_{t+1})),\quad t\in \mathbb{T}\setminus\{T\},\, n\in\mN,
\end{align*}
where  \(B = S \times W\) and 
\begin{equation*}
	\pi^n_{t+1}(x,u)=\pi_{t+1}(x,u;\alpha^n_t(u),m^n_t(u);dx',du'),\, \pi_{t+1}(x,u)=\pi_{t+1}(x,u;\alpha_t(u),m_t(u);dx',du')
\end{equation*}
with $\alpha^n\rightarrow\alpha$ and $m^n\rightarrow m$.

Thus, \((X, U)\) and \((X^n, U^n)\) are \(\mathbb{F}\)-Markov processes with transition kernels \(\pi\) and \(\pi^n\), respectively. We can show by induction that for any \(k \in \mathbb{T}\), \((X^n_k, U^n_k) \to (X_k, U_k)\) \(\mathbb{P}^{c}\)-almost surely. The base case \(k = 0\) is obvious. Assume  that \((X^n_k, U^n_k) \to (X_k, U_k)\) \(\mathbb{P}^{c}\)-a.s.  By Remark \ref{remark:pi}, Lemma \ref{lemma:a1}, and the induction hypothesis, we conclude that \((X^n_{k+1}, U^n_{k+1}) \to (X_{k+1}, U_{k+1})\) \(\mathbb{P}^{c}\)-a.s, which verifies the claim.

\vspace{0.3in}\noindent
\textbf{Acknowledgement}: The authors are grateful to two anonymous referees for their helpful comments
and suggestions. X. Yu and K. Zhang are supported by the Hong Kong Polytechnic University research Grant under No. P0045654 and by the Research Centre for Quantitative Finance at the Hong Kong Polytechnic University under grant No. P0042708.

	\bibliographystyle{siam}
	{\small
	\bibliography{ref}
	}
	
\end{document}